\documentclass[10pt,reqno]{article} 
\usepackage{amsfonts,amsmath,layout}
\usepackage{mathrsfs}  
\usepackage{soul}
\usepackage{amssymb,mathabx,amsfonts,amsthm,amscd,stmaryrd,dsfont,esint,upgreek,constants,todonotes}
\DeclareFontFamily{OT1}{pzc}{}
\DeclareFontShape{OT1}{pzc}{m}{it}{<-> s * [1.10] pzcmi7t}{}
\DeclareMathAlphabet{\mathpzc}{OT1}{pzc}{m}{it}

\usepackage{color} 

\definecolor{Red}{cmyk}{0,1,1,0.2}

\newcommand{\N}{\mathbb N}

\newcommand{\R}{\mathbb R}

\def\R{\mathbb R}
\def\N{\mathbb N}

\def\ep{\epsilon}

\newcommand{\be}{\begin{equation}}
\newcommand{\ee}{\end{equation}}
\def\1{{\bf 1}}

\def\ds{\displaystyle}

\newtheorem{Theorem}{Theorem}[section]
\newtheorem{Definition}[Theorem]{Definition}
\newtheorem{Proposition}[Theorem]{Proposition}
\newtheorem{Lemma}[Theorem]{Lemma}
\newtheorem{Corollary}[Theorem]{Corollary}
\newtheorem{Remark}[Theorem]{Remark}

\begin{document}


\title{A microscopic derivation of a traffic flow model on a junction with two entry lines}
\author{P. Cardaliaguet}
\author{\renewcommand{\thefootnote}{\arabic{footnote}}
  P. Cardaliaguet\footnotemark[1]}
\footnotetext[1]{CEREMADE, UMR CNRS 7534, Universit\'e Paris Dauphine-PSL,
Place de Lattre de Tassigny, 75775 Paris Cedex 16, France. }

\maketitle

\begin{abstract} We derive a conservation law on a network made of two incoming branches and a single outgoing one from a discrete traffic flow model. The continuous model is obtained from the discrete one by letting the number of vehicles tend to infinity and after scaling. In the discrete model, the vehicles  solve a follow-the-leader model on each branch. The priority rule at the junction is given by a time-periodic traffic light (with a period tending to infinity with the scaling for technical reasons). The conservation law at the limit is $L^1-$contractive and described in terms of a germ depending on the flux and on the average amount of time the traffic light is ``green'' on a branch. 
\end{abstract}



\section{Introduction} 

\subsection{Foreword}

Macroscopic models of traffic flows on networks have attracted a lot of attention in the recent years: see for instance the monographs or survey papers \cite{BDMGGP22, BCGHP, GHP16}, and the references therein. If the model in the interior of each branch of the netwok makes a relative consensus---it is very often, in a first approximation, a conservation law known as the LRW model---, the conditions to be put on the junction is more discussed. Several possible rules have been proposed (see the references above), at least when the junction involves more than three roads. The goal of the present paper is to pin-point one specific junction rule---on a junction with two incoming roads and a single out-going one---which pops up after homogenization of a microscopic model. In our microscopic model (a follow-the-leader model on each branch), the traffic flow is regulated at the junction by a traffic light, preventing the vehicles to collide. The main interest of the paper is that it can serve as a justification of the continuous limit model. As far as we know, this is the first micro-macro justification of a continuous model involving two incoming roads. Note however that we assume that the periodic of the traffic light tends to infinity after scaling, an assumption that we discuss below.

Before going further, let us first recall that the micro-macro derivation of a conservation law (the LWR model in traffic flow) from a follow-the-leader model on a single line is classical \cite{MR1952895, CFGG, rigorousLWR, GoatinRossi, MR3721873}. Continuous models involving a junction separating 2 half-roads (the so-called 1:1 junction)  has been the subject of a very large literature that we do not try to review here. It is  known that the junction condition then reduces to a single flux limiter (a real number) \cite{AGDV11, AKR11, IM17, MISHRASurvey}, and the derivation of this junction condition from microscopic model has been studied thoroughly \cite{FSZ18, FoSa20}. Moreover, if there are several approaches to describe the continuous model (for instance through germs \cite{AKR11}, or through Hamilton-Jacobi equation \cite{IM17}), these approaches are now known to be equivalent: see  \cite{CFGM} for convex fluxes and \cite{FIM24} for nonconvex ones, and the references therein. 

The situation is completely different for junctions involving more than 3 roads. The main issue is that the number of possible models is huge and not restricted to a single flux limiter. In addition, approaches by conservation law with a germ condition \cite{FMR22b} and by flux limited Hamilton-Jacobi \cite{IM17} are no longer equivalent in general: see the counter-example in \cite{CFGM}.  So, in terms of modeling, the junction condition should certainly be dictated by physical considerations, as, for instance, by an underlying microscopic model. A first example in this direction is \cite{CFarma}, in which we study the micro-macro derivation of a traffic flow involving one incoming road and two out-going ones (the 1:2 junction). In this model, a fixed portion of the incoming traffic (understood in statistical terms) is dispatched in each outgoing road.  We prove in \cite{CFarma} a homogenization result, in which the limit model can be described as a flux limited Hamilton-Jacobi equation on the network. The corresponding system in terms of conservation law is, however, poorly understood up to now. As explained in \cite{CFGM}, it is {\it not} an $L^1-$contraction in general. 

The situation of an 2:1 junction is even more subtle, as it is not too natural (at least for a traffic flow model) to prescribe in advance the proportion of vehicle coming from a given incoming branch. This seems to exclude an approach by Hamilton-Jacobi equation. We introduce in \cite{CFM24} a family of junction conditions on this 2:1 network, which turns out to be $L^1-$contractive. In addition, still in  \cite{CFM24}, we derive some of these models from mesoscopic ones, in which the traffic flow is continuous, but the traffic lights discrete. The method of proof relies on standard BV estimates, on the $L^1$ contraction property of the mesoscopic models and on the construction of correctors. However we left open the more difficult derivation of these models from microscopic ones: indeed, in this discrete setting, BV estimates, $L^1$ contraction and the construction of correctors seem very challenging questions in general. The aim of the present paper to solve---at least partially and in a particular case--this issue.

Let us now describe our microscopic and macroscopic models. We consider a junction involving two in-going roads and an outgoing one. For the microscopic model, we use a classical follow-the-leader equation, which holds on each roads ``far'' from the junction:
$$
\dot X^i(t)= V(X^{\sigma(i)}(t)-X^i(t)), \qquad i\in \{1, \dots, N\}.
$$
Here $N$ is the number of vehicles, the velocity rule $V:\R_+\to \R_+$ is a given increasing and bounded map and $X^{\sigma(i)}(t)$ denotes the position of the vehicle in front of vehicle $i$ at position $X^i(t)$. To avoid collisions at the junction, a time-periodic traffic light blocks one of the two in-coming roads alternatively, the other one evolving freely: see the next subsection for details. Our main result (Theorem \ref{thm.main}) states that, when the number $N$ of particle tends to infinity and after a hyperbolic scaling, the empirical density associated with the solution of the above follow-the-leader models converges---in a suitable sense---to the solution of the following conservation law: 
$$
\begin{array}{lllll}
(i) & \rho^k\in [0,\rho_{\max}]&\qquad \text{ a.e. on}\; &(0,\infty)\times \mathcal R^k, &\; k=0,1,2,\\
(ii) & \partial_t \rho^k +(f(\rho^k))_x= 0 &\qquad \text{ on}\; &(0,\infty)\times \mathcal R^k, &\; k=0,1,2,\\
(iii) & (\rho^0(t,0^+), \rho^1(t,0^-), \rho^2(t,0^-))\in \mathcal G& \qquad \text{ for a.e.}\;  &t\in (0,\infty),&
\end{array}
$$
Here $\mathcal R^k=(-\infty, 0)\times \{k\}$ for $k=1,2$ are the incoming roads, $\mathcal R^0=(0,\infty)\times \{0\}$ is the outgoing one, the junction being at $x=0$. The flux $f: [0,\rho_{\max}]\to \R_+$ is a smooth uniformly concave map, with $f(0)=f(\rho_{\max})=0$. It is  related to the velocity law $V$ by the equality $f(r)=rV(1/r)$. The notation $ (\rho^0(t,0^+), \rho^1(t,0^-), \rho^2(t,0^-))$ stands for the strong trace of the map  $(\rho^0, \rho^1, \rho^2)$ at $x=0$, which is known to exist \cite{Ps07}. Finally, the homogeneous germ $\mathcal G$, given explicitly in formula \eqref{def.mathG} below, keeps track of the traffic light model and explains how the outgoing density is related to the incoming ones at the junction. Our main assumption on the model is that the period of the traffic light is quite long (in terms of scaling): we make this (unnatural) assumption for two reasons: the first one is that it allows the discrete model to be close of a mesoscopic model studied in \cite{CFM24}, and thus to identify the germ $\mathcal G$ explicitly.  The second reason for our assumption is that it allows some regularization property of the discrete equation to take place. We explain these two points in detail  now.   \\

The main difficulty  of our problem (compared for instance to the micro-macro derivation on 1:1 junctions  \cite{FSZ18, FoSa20} or on 1:2 junctions \cite{CFarma}) is the fact that we cannot work at the Hamilton-Jacobi level: the limit model has to be described through the germ condition, which requires---as far as we know---much finer analysis and estimates. For instance, the techniques of proof in  \cite{CFarma, FSZ18, FoSa20} rely only on $L^\infty$ bounds on the ``density'' of the discrete model in order to pass to the limit. Here we will need BV bounds and approximate Kato's inequality to do so. To overcome this issue, we borrow  several argument from the paper by Di Francesco and Rosini \cite{rigorousLWR}. Let us recall that  \cite{rigorousLWR} discusses the convergence of the follow-the-leader model to the corresponding conservation law on a simple road. For this it establishes a  BV regularization effect on the discrete solution, the fact that the discrete solution is an approximate entropy solution to the conservation law and, therefore, satisfies an approximate Kato's inequality.  We adapt these results to a follow-the-leader model stopped by  a red light (Propositions \ref{prop.BoundBVRed} and \ref{prop.rateFeu}) and then to the whole system (Proposition \ref{prop.gobBV} and Proposition \ref{prop.globalesti}). Unfortunately, the traffic light creates shocks which brutally increase the total variation of the discrete density. One of the reasons why we have to ask for long intervals between changes in the traffic light is precisely to compensate this deterioration, by using the regularization effect of \cite{rigorousLWR}. Another difficulty of our problem is the identification of the correct germ. For this, following a classical approach in homogenization, we need to build suitable correctors for the problem. This step is in general quite challenging.  However, in our setting where the period of the traffic light becomes larger and larger, we can almost argue as if we were in  a mesoscopic model in which the traffic flow is continuous, but the traffic light discrete.  For such a model, correctors have been built in  \cite{CFM24} and we use this construction to prove ultimately that the limit is an entropy solution at the junction. \\

{\bf Organization of the paper:} In the rest of the section, we introduce our main assumptions (Subsection \ref{subsec.hyp}), explicit the discrete model (Subsection \ref{subsec.micro}) and the continuous one (Subsection \ref{subsec.conti}) and finally give our main convergence result (Subsection \ref{subsec.main}). Section \ref{sec.prelim} is devoted to basic existence and properties of the discrete model. In Section \ref{sec.Esti1branch}, we first recall the main results of \cite{rigorousLWR} for scaled follow-the-leader models on a whole line, adapt them to models with a traffic light and finally state a local compactness result for scaled follow-the-leader models.  The proof of the main result is then given in Section \ref{sec.proofmain}.  

\subsection{Notation and assumptions} \label{subsec.hyp}

We collect here our main notation and assumptions.  The network is made of two incoming roads and a single outgoing one, meeting at the junction point at the origin $0$. We let $\mathcal R^j=(-\infty,0)\times \{j\}$ (for $j=1,2$) be the incoming branches, $\mathcal R^0= (0,\infty)\times \{0\}$  being the outgoing one. We consider the set $\mathcal R=\bigcup_{j=0}^2 \mathcal R^{j}\cup\{0\}$ with the topology of three half lines glued together at the origin $0$.\\

{\bf Assumption (H).} The flux $f: [0, \rho_{\max}]\to  \R_+$ of the continuous conservation law is  of class $C^2$ and uniformly concave, with $f(0)=f(\rho_{\max})=0$. We denote by $V:\R_+\to\R_+$ the velocity of the discrete model. It is defined  by $V(e)= ef(1/e)$ for $e\in [e_{\min},\infty)$, where $e_{\min}:= 1/\rho_{\max}$,  and $V(e)= 0$ on $[0,e_{\min}]$. We note that $V$ is Lipschitz continuous, nondecreasing and bounded on $\R_+$, and of class $C^2$ on $(e_{\min},\infty)$. We set $V_{\max}= \max V= f'(0)$. Finally, setting $v(r)=V(1/r)$, we  assume that $r\to rv'(r)$ is  nonincreasing on $(0,\rho_{\max})$. \\

{\it Example.} For instance, if $f(r)= Ar-Br^2$ (for $A,B>0$) and $\rho_{\max}= AB^{-1}$, then $v(r)= A-Br$ is defined on $[0, AB^{-1}]$ and is such that $r\to rv'(r)= -Br$ is nonincreasing, while $V$, defined by $V(x)= 0$ on $[0,e_{\min}]$ with $e_{\min}= A^{-1}B=1/\rho_{\max}$ and $V(x)= A-Bx^{-1}$ if $x>e_{\min}$, is nondecreasing and bounded.

\subsection{The microscopic model} \label{subsec.micro}

In the discrete (or microscopic) model, there are $N$ vehicles. The vehicles positions are encoded by the family $(X^i(t), r^i(t))_{i=1, \dots, N}$ in which: $r^i(t)\in \{0,1,2\}$ is the label of the road on which Vehicle $i$ is standing at time $t$ and $X^i(t)\in \R$ is the position of the car on the road. By definition, $r^i(t)=0$ if and only if $X^i(t)\geq 0$. We  assume that the initial condition $(X^i_0, r^i_0)$ satisfies: 
\be\label{cond.init}
X^i_0 \neq X^j_0 \qquad \text{if}\; i\neq j\;\text{and}\; r^i_0  =r^j_0.
\ee
We  check below that this  properties propagates in time. We define the index $\sigma^i(t)$ of the vehicle  in front of vehicle $i$ at time $t$ as: 
$$
\sigma^i(t)=  \left\{\begin{array}{ll}
 {\rm argmin}_j  \{X^j(t), \; X^j(t)>X^i(t)\} & \text{if} \; X^i(t)\geq 0,\\
 {\rm argmin}_j \{X^j(t),\; X^j(t)>X^i(t) \;\text{and} \; r^j(t)\in \{r^i(t),0\} \}& \text{if} \; X^j(t)<0.\end{array}\right.
$$
We set $\sigma^i(t)= \infty$ if there is no such a $j$.  Then by convention $X^\infty(t)=\infty$. We denote by $T>0$ the period of the traffic light and assume that the traffic is green for Road 1 on intervals of the form $(0,T_1)\; {\rm mod}\; [T]$ and red otherwise, where $T_1\in (0,T)$. The traffic light is viewed as a vehicle standing at $x=0$ when it is red. This yields to a dynamic of the form 
\be\label{eq.X}
\dot X^i(t)= 
\left\{ \begin{array}{ll}
V( -X^i(t)) & \text{if}\; \left\{\begin{array}{l} r^i (t) =1,\; t\in [T_1,T)\; {\rm mod}\; [T] \; \text{and} \;  X^{\sigma^i(t)}(t)\geq 0,\\
\text{or}\; r^i (t) = 2 ,\; t\in [0, T_1)\; {\rm mod}\; [T]\;\text{and} \;  X^{\sigma^i(t)}(t)\geq 0,\end{array}\right. \\
V( X^{\sigma^i(t)}(t)-X^i(t)) & \text{otherwise}.
\end{array}\right.
\ee
{\it Explanation for \eqref{eq.X}:} In \eqref{eq.X}, the first conditions take into account the traffic light: if $t\in [T_1,T)\; {\rm mod}$, the traffic light is red on Road 1, and the condition $ r^i (t) =1$ and $X^{\sigma^i(t)}(t)\geq 0$ mean that the vehicle of index $i$ is the right-most vehicle on Road 1; for this reason it is the one which has to take into account the traffic light---seen as a vehicle standing at $x=0$ and has therefore a velocity $V(0-X^i)$. The symmetric explanation applies when $t\in [0, T_1)\; {\rm mod}\; [T]$ on which the traffic light is red on Road 2. In all the other cases, the vehicle follows a standard follow-the-leader model. 

In addition, we suppose that the initial condition holds 
\be\label{eq.init}
(X^i(0), r^i(0))= (X^i_0, r^i_0)\qquad i=1, \dots, N,
\ee
and the $r^i$ are defined by 
\be\label{eq.defri}
r^i(t)= r^i_0 \; \text{as long as } X^i(t)\neq 0, \; \text{and $r^i(t)= 0$ if and only if $X^i(t)\geq 0$.}
\ee
By a solution  $(X^i, r^i)_{i=1, \dots, N}$ to  \eqref{eq.X}-\eqref{eq.init}-\eqref{eq.defri}, we mean that the $X^i$ are Lipschitz continuous, with  $\dot X^i$ and  $r^i$  c\`{a}dl\`{a}g, and such that
\be\label{cond.initpropag}
X^i(t) \neq X^j(t) \qquad \text{if}\; i\neq j\;\text{and}\; r^i(t) =r^j(t),
\ee
and that  \eqref{eq.X} holds a.e. and \eqref{eq.init} and \eqref{eq.defri} hold. We note that this implies that the $\sigma^i$ are c\`{a}dl\`{a}g as well. 

\begin{Proposition}\label{prop.existidisc} Under our standing assumption {\bf (H)}, given an initial condition  $(X^i_0, r^i_0)_{i=1, \dots, N}$ satisfying \eqref{cond.init}, there exists a unique solution  $(X^i, r^i)_{i=1, \dots, N}$ to  \eqref{eq.X}-\eqref{eq.init}-\eqref{eq.defri}. 
\end{Proposition} 

The proof of Proposition \eqref{prop.existidisc} is given is Section \ref{sec.prelim}. 

\subsection{The limit model: a conservation law on the network} \label{subsec.conti}

We define here the continuous model. It is a conservation law on the network with a homogeneous germ $\mathcal G$ at the junction. It  takes the form:  
\be\label{eq.macro}
\begin{array}{lllll}
(i) & \rho^k\in [0,\rho_{\max}]&\qquad \text{ a.e. on}\; &(0,\infty)\times \mathcal R^k, &\; k=0,1,2,\\
(ii) & \partial_t \rho^k +(f(\rho^k))_x= 0 &\qquad \text{ on}\; &(0,\infty)\times \mathcal R^k, &\; k=0,1,2,\\
(iii) & (\rho^0(t,0), \rho^1(t,0), \rho^2(t,0))\in \mathcal G& \qquad \text{ for a.e.}\;  &t\in (0,\infty),&
\end{array}
\ee
In order to define the notion of solution of \eqref{eq.macro}, let us first  recall the notion of germ on a junction. Following \cite{AKR11,MFR22}, we say that the pair (Kru\u{z}kov entropy, entropy flux) is given, for $p,\bar p \in \R$, by  
$$
\eta(\bar p,p)=|p-\bar p| , \qquad q(\bar p,p) = \mbox{sign}(p-\bar p)(f(p)-f(\bar p)).
$$
In the box $Q:=[0,\rho_{\max}]^3$, we let $Q^{RH}$ be the subset of $Q$ satisfying Rankine-Hugoniot condition:
\begin{equation}\label{eq::e**2}
Q^{RH}:=\left\{P=(p^0,p^1,p^2)\in Q,\quad f(p^0)=f(p^1)+f(p^2)\right\}
\end{equation}
Given  $P=(p^0,p^1,p^2)$, $\bar P=(\bar p^0,\bar p^1,\bar p^2) \in Q$, we define the dissipation by
$$
D(\bar P,P):= \left\{q^1(\bar p^1,p^1)+q^2(\bar p^2,p^2)\right\} - q^0(\bar p^0,p^0)=\mbox{IN}-\mbox{OUT}
$$
A set $G\subset Q$ is said to be a germ  if it satisfies
$$\left\{\begin{array}{ll}
G\subset Q^{RH}&\quad \mbox{(Rankine-Hugoniot)}\\
D(\bar P,P)\ge 0\quad \mbox{for all}\quad \bar P,P\in G &\quad \mbox{(dissipation)}\\
\end{array}\right.$$
We say that $G$ is maximal if for every $P\in Q$, we have
$$\left(D(\bar P,P)\ge 0\quad \mbox{for all}\quad \bar P\in G\right)\quad \Longrightarrow\quad P\in G.$$

For our problem, the limit germ ${\mathcal G}$ at the junction is defined as follows (see Example~1 and Proposition 2.6 in \cite{CFM24}).  Let $\theta=T_1/T\in (0,1)$ be the proportion of time the traffic-light is green on Road 1. We set $\theta^0=1$, $\theta^1=\theta$ and $\theta^2=1-\theta$. Then $\mathcal G$ is given by 
\be\label{def.mathG}
{\mathcal G}:=\left\{P=(p^0,p^1,p^2)\in Q^{RH},\quad \left|\begin{array}{ll}
0\le f(p^k)\le \theta^k \max f ,&\quad k=0,1,2\\
\\
f^-(p^k)\ge \theta^k f^-(p^0),&\quad k=1,2\\
\end{array}\right.\right\},
\ee
where $f^-$ is the smallest nonincreasing map above $f$. 
Following \cite[Theorem 2.1 and Proposition 2.6]{CFM24}, $\mathcal G$ is a maximal germ generated by the set 
\be\label{def.E}
E=\Gamma \cup \{P_0, P_1, P_2, P_3\}
\ee
where
\be\label{def.Gamma}
\Gamma = \{((f^{-})^{-1}(\lambda), (f^{-})^{-1}(\theta_1 \lambda), (f^{-})^{-1}(\theta_2\lambda)), \; \lambda \in [0, \max f]\}.
\ee
 and 
\begin{align*}
&P_1= ((f^{+})^{-1}(\theta_1 \max f), (f^{-})^{-1}(\theta_1\max f), 0), \\
&P_2=  ((f^{+})^{-1}(\theta_2 \max f),0,  (f^{-})^{-1}(\theta_2\max f)), \\ 
& P_3 =(0,0,0) , \qquad P_0= ((f^{-})^{-1}(0), (f^{-})^{-1}(0), (f^{-})^{-1}(0))= (\rho_{\max},\rho_{\max},\rho_{\max}) \in \Gamma .
\end{align*}
This means that 
$$
\left(D(\bar P,P)\ge 0\quad \mbox{for all}\quad \bar P\in E\right)\quad \Longrightarrow\quad P\in \mathcal G.
$$


\begin{Definition}{\bf (Entropy solution of (\ref{eq.macro}))}\label{def.solwithgerm}\\ 
Given a maximal germ $\mathcal G\subset Q$ and an initial condition $\bar \rho\in L^\infty(\mathcal R)$ such that $\bar \rho^k\in [0, \rho_{\max}]$ a.e. for $k=0,1,2$, we say that a map  $\rho\in L^\infty((0,\infty)\times \mathcal R)$ is an entropy solution of \eqref{eq.macro} if, for any $k=0,1,2$, $\rho^k$ is a Kruzkhov entropy solution of \eqref{eq.macro}-(ii) on $\mathcal R^j$, if {its trace at $t=0$ is $\bar \rho$} and if its trace $\rho(\cdot,0)=(\rho^0(\cdot,0^+),\rho^1(\cdot,0^-),\rho^2(\cdot,0^-))$ at $x=0$ belongs to $\mathcal G$: 
$$
\rho(t,0)\in \mathcal G\qquad a.e. \; t\geq0.
$$
\end{Definition} 

For the germ $\mathcal G$ given by \eqref{def.mathG} we proved in \cite{CFM24} that there exists a unique solution to \eqref{eq.macro}. 

\begin{Remark} \label{rem.entropsol} Because the germ $\mathcal G$ is maximal and generated by the set $E$, Definition \ref{def.solwithgerm} is equivalent to the following entropy inequality (see  \cite{FMR22b,MFR22}): 
$$
\sum_{j=0}^2\left\{ \int_0^\infty \int_{\mathcal R^j} \eta(u^j,\rho^j)\partial_t \phi^j+ q^j(u^j,\rho^j) \partial_x\phi^j +\int_{\mathcal R^j} \eta(u^j,\bar \rho^j)\phi^j(0,x)\right\} \geq 0
$$
for any $u=(u^j) \in E$ and any continuous nonnegative test function $\phi:[0,\infty)\times \mathcal R\to [0,\infty)$ with a compact support and such that $\phi^j:=\phi_{|[0,+\infty)\times (\mathcal R^j\cup \left\{0\right\})}$ is $C^1$ for any $j=0,1,2$. 
\end{Remark}

\subsection{Main result}\label{subsec.main}

Let $(X^i,r^i)$ be the solution of the model with traffic light defined in Subsection \ref{subsec.micro}, i.e., satisfy \eqref{eq.X}-\eqref{eq.init}-\eqref{eq.defri}. We set 
$$
x^{\ep,i}(t)= \ep X^i(\ep^{-1}t)
$$
and 
$$
y^{\ep,i}(t) = \frac{1}{X^{\sigma^i(\ep^{-1}t)}(\ep^{-1}t)-X^i(\ep^{-1}t)} = \frac{\ep}{x^{\ep, \sigma^i(\ep^{-1}t)}(t)-x^{\ep,i}(t)}.
$$
We define 
\be\label{def.rho}
\rho^{\ep,k}(t,x)
=
\left\{\begin{array}{ll} 
\ds  \sum_{r^i(\ep^{-1}t) =k} y^{\ep, i}(t)  {\bf 1}_{[x^i(t), x^{\ep,\sigma^i(\ep^{-1}t)}(t))} & \text{if}\; k=1, 2, \\
\ds  & \\
\ds   \sum_{r^i(\ep^{-1}t) = 0} y^{\ep, i}(t)  {\bf 1}_{[x^i(t), x^{\ep,\sigma^i(\ep^{-1}t)}(t))} & \text{if}\; k=0.
\end{array}\right.
\ee
Note for later use that $\rho^{\ep,k}\rho^{\ep,0}=0$ for $k=1,2$. Let us warn the reader that the $\rho^{\ep, k}_X$ are discontinuous in time (even in a weak topology). Discontinuities occur as one of the vehicle reaches the junction $x=0$. 

\begin{Theorem}\label{thm.main} Suppose, in addition to our standing assumption ${\bf (H)}$, that 
\be\label{Nep=O1}
N\ep =O(1), 
\ee
that the initial condition satisfies \eqref{cond.init} and
\be\label{hyp.condinti}
X^{\sigma^i(0)}_0\geq X^i_0 +\rho_{\max}^{-1}\qquad \forall i\in \{1, \dots, N-1\},
\ee
with a BV bound and a uniformly compact support:
\be\label{BVt=0}
TV(\rho^{\ep,k}(0,\cdot))\leq C_0,\qquad \text{Spt}(\rho^{\ep,k}(0,\cdot))\subset [-C_0,C_0] \qquad \text{for}\; k=0,1,2, 
\ee
and that there exists $\bar \rho\in L^\infty(\mathcal R)$ such that 
\be\label{limt=0}
\rho^{\ep}(0,\cdot) \to \bar \rho \qquad \text{in}\; L^1(\mathcal R)\qquad \text{as }\; \ep\to 0^+.
\ee
Assume also that the traffic light is such that there exists $\alpha\in (2/3,1)$ and $\theta\in (0,1)$ with 
\be\label{hyp.T}
T= \ep^{-\alpha} \qquad \text{and}\qquad T_1 = \theta T.
\ee
Then  $(\rho^{\ep,k})_{k=0,1,2}$ converges in $L^1_{loc}$ to the solution $(\rho^k)_{k=0,1,2}$ of \eqref{eq.macro}, for the germ $\mathcal G$ given by \eqref{def.mathG}. 
\end{Theorem} 

The proof requires several preliminary  estimates and is given at the end of Section \ref{sec.proofmain}. Note that  the $(X^i,r^i)$---even before scaling---actually depend on $\ep$ because the period of the  traffic light depends on $\ep$ (in assumption \eqref{hyp.T}). We do not make this dependence explicit for simplicity of notation. 

\section{Preliminary results} \label{sec.prelim}

In this section, we discuss some basic existence results and estimates for the discrete system.
\subsection{Existence of a solution for the discrete model}

Here we prove Proposition \ref{prop.existidisc} and show various basic properties of the microscopic model. 

\begin{proof}[Proof of Proposition \ref{prop.existidisc}] The difficulty to prove the well-posedness of the solution of the discrete model is the discontinuity of the dynamics (through the discontinuity of the $r^i$ and the $\sigma^i$ with respect to the position the vehicles). Discontinuities occur when one of the vehicles crosses the junction $x=0$. 

In view of assumption \eqref{cond.init} on the initial condition, the local in time existence and uniqueness  is straightforward, since the maps $r^i$ and $\sigma^i$ remain constant on a short time interval after $t=0$.  For the proof of global existence and uniqueness, as well as for  \eqref{cond.initpropag}, we argue by induction on time intervals of the form $[nT, (n+1)T]$. Let us assume that the solution exists and is unique on $[0,nT]$ and satisfies \eqref{cond.initpropag} up to time $nT$. Then there exists a maximal interval $[nT, nT+\tau)\subset [nT, nT+T_1]$ on which the solution exists and satisfies \eqref{cond.initpropag}. Let us check that $\tau= T_1$ and that \eqref{cond.initpropag} holds at $nT+T_1$. For this, let $i_2={\rm argmax}_i \{X^i(nT), \; r^i(nT)=2\}$. We  note that $r^{i_2}(t)= 2$ on $[nT, nT+\tau)$ since $X^{i_2}(Nt)<0$ and $\dot X^{i_2}(t)= V(-X^{i_2}(t))$ on this interval, with $V$ Lipschitz and $V(0)=0$. Therefore the $\sigma^i$ remain constant on $[nT, nT+\tau)$ for any  $i\neq i_2$, and thus the solution is unique on that interval. On $[nT, nT+\tau)$, we claim that 
\be\label{maueqk:jsfdgj}
\inf_{\{i, \; r^i(t)\in \{0,1\}\}} (X^{\sigma^i(t)}(t)-X^i(t)) \geq \inf_{\{i, \; r^i(t)\in \{0,1\}\}} (X^{\sigma^i(nT)}(nT)-X^i(nT)) =:\alpha_1>0. 
\ee
This is a  consequence of  \cite[Lemma 3.1]{rigorousLWR} or \cite[Proposition 4.1]{rigorousLWRbis} and the fact that, on this time interval, the indices $\sigma^i$ remain constant and the family $(X^i)_{ \{i, \; r^i(t)\in \{0,1\}\}}$ solves the unconstrained problem
\be\label{eq.unconstr??}
\dot X^i(t) = V(X^{\sigma^i(t)}(t)-X^i(t)), 
\ee
(with convention $\dot X^i(t)=V_{\max}$ if $\sigma^i(t)=\infty$). Let us now consider the family $(X^i)_{\{i, \; r^i(nT)=2\}}$. Let us set 
$$
X^i(t) = (i-(N+1)) \rho_{\max}^{-1}, \qquad i\in \N,\; i\geq N+1, \; t\geq nT. 
$$
We note that, on the time-interval $[nT, nT+\tau)$, the extended family $(X^i)_{\{i, \; r^i(nT)=2\; {\text or} \; i\geq N+1\}} $ solves an unconstrained problem of the form \eqref{eq.unconstr??} because $V(\rho_{\max}^{-1})=0$. Thus we can apply again \cite[Lemma 3.1]{rigorousLWR} and infer that, on $[nT, nT+\tau)$,  
\be\label{maueqk:jsfdgj2}
\inf_{\{i, \; r^i(t)=2, \; i\neq i_2\}} (X^{\sigma^i(t)}(t)-X^i(t)) \geq \inf_{\{i, \; r^i(nT)=2 \; {\text or} \; i\geq N+1\}} (X^{\sigma^i(nT)}(nT)-X^i(nT)) =:\alpha_2>0. 
\ee
In the particular case of $i_2$, one gets the sharper estimate 
\be\label{maueqk:jsfdgj3}
-X^{i_2}(t)\geq \alpha_2
\ee
because the vehicle in front of $X^{i_2}$ is $X^{N+1}\equiv 0$. 
We now show that \eqref{maueqk:jsfdgj}, \eqref{maueqk:jsfdgj2}  and \eqref{maueqk:jsfdgj3} imply that  \eqref{cond.initpropag} holds up to time $nT+\tau$. Indeed, if not, there exists  $i\neq j$ with $X^i(\tau) = X^j(\tau)$ and $r^i(\tau) =r^j(\tau)$. As \eqref{cond.initpropag} holds before $nT+\tau$, we can assume to fix the ideas that there exists $t_k\to \tau^-$ such that $X^i(t_k)<X^j(t_k)$ for all $k$. As $r^i$ and $r^j$ are c\`{a}dl\`{a}g, we can also assume that $r^i(t_k)$ and $r^j(t_k)$ do not depend on $k$. Note first that, if $r^i(t_0)\in \{0,1\}$ and $r^j(t_0)\in \{0,1\}$, or if $r^i(t_0)=r^j(t_0)=2$, then, by the definition of $\sigma^i$ and \eqref{maueqk:jsfdgj} or \eqref{maueqk:jsfdgj2} (note that in this case, $i\neq i_2$), 
$$
X^j(t_k)- X^i(t_k)\geq X^{\sigma^i(t_k)}(t_k)- X^{i}(t_k) \geq   \alpha_1\wedge \alpha_2>0.
$$
Letting $k\to \infty$ leads to a contradiction with the equality  $X^i(\tau) = X^j(\tau)$. We now assume that $r^i(t_0)=2$ and $r^j(t_0)\in \{0,1\}$. As $r^i(\tau) =r^j(\tau)$, this means that $X^i(\tau) = X^j(\tau)=0$. But then
$$
-X^i(t_k)\geq -X^{i_2}(t_k) \geq \alpha_2>0, 
$$
which leads to a contradiction with the fact that $X^i(\tau) =0$ by letting $k\to\infty$. The opposite case $r^i(t_0)\in \{0,1\}$ and $r^j(t_0)=2$ is similar. Therefore \eqref{cond.initpropag} holds up to time $nT+\tau$, so that $\tau=T_1$ holds. We can argue symmetrically on the time interval $[nT+T_1, (n+1)T]$ and conclude by induction. 
\end{proof}

\subsection{Uniform bound on $\rho^\ep_X$} 

Next we show that $\rho^\ep_X$ is bounded by $\rho_{\max}$ far from the junction. Let us start with a remark. 

 \begin{Lemma}\label{lem.x+1-xpetit}  For any $t$ and $i$, if $X^{\sigma^i(t)}(t)-X^i(t) < \rho_{\max}^{-1}$, then $X^i(t)<0$ and  $X^{\sigma^i(t)}(t)\geq 0$. 
 \end{Lemma}
  
 \begin{proof} Assume on the contrary that there exists a time $t\geq 0$ such that $X^{\sigma^i(t)}(t)-X^i(t) < \rho_{\max}^{-1}$ and either  $X^i(t)\geq 0$ or  $X^{\sigma^i(t)}(t)< 0$. Note that, by assumption \eqref{hyp.condinti} on the initial condition, $t$ is positive. 
 
  Let us first assume that $X^{\sigma^i(t)}(t)< 0$. Then, for any $s\in [0,t]$,  $\sigma^i(s)=\sigma^i(t)$ and $X^{\sigma^i(s)}(s)< 0$. Let $(a,t]$ be the largest interval on which $X^{\sigma^i(s)}(t)-X^i(s) < \rho_{\max}^{-1}$. As $\dot X^i(s) = V(X^{\sigma^i(s)}(t)-X^i(s))=0$ on $(a,t]$, we have $X^i(a)=X^i(t)$ and thus  
 $$
  X^{\sigma^i(a)}(a)-X^i(a) \leq X^{\sigma^i(t)}(t)-X^i(t)  <  \rho_{\max}^{-1}. 
 $$
Hence $a=0$, which contradicts our assumption \eqref{hyp.condinti}. 
 
 We now assume that $X^i(t)\geq 0$. We first prove that $\sigma^i(t^-)= \sigma^i(t)$. Indeed, if there is $\ep>0$ such that $X^i(s)\geq 0$ in $[t-\ep, t]$,  this is obvious because $\sigma^i=0$ on $[t-\ep, t]$. Otherwise, $X^i(s)<0$ on $[0,t)$ and $X^i(t)=0$. Note that $X^{\sigma^i(t)}(t)>0$ by definition, so that there exists $\ep>0$  small enough such that $X^{\sigma^i(t)}(s)>0$ for any $s\in [t-\ep, t]$. We also assume that $\sigma^i(s)= \sigma^i(t^-)$ on $[t-\ep,t]$, which is possible for $\ep>0$ small enough since $\sigma^i$ is c\`{a}dl\`{a}g. As $X^i(s)<0$ and $X^i(t)=0$, $\dot X^i\not\equiv 0$ on $[t-\ep, t]$ and thus there exists $s_n\to t^-$ such that $X^{\sigma^i(t^-)}(s_n)-X^i(s_n)> \rho_{\max}^{-1}$. Then, for $n$ large, $X^{\sigma^i(t^-)}(s_n)\geq 0$. Thus $X^{\sigma^i(t^-)}(s_n)\leq X^{\sigma^i(t)}(s_n)$ by definition of $\sigma^i$. This inequality propagates up to time $t$ since the equation is order preserving on $\mathcal R^0$:  $0\leq X^{\sigma^i(t^-)}(t)\leq X^{\sigma^i(t)}(t)$. But $X^{\sigma^i(t)}(t)\leq X^{\sigma^i(t^-)}(t)$ by the definition of $\sigma^i$. By \eqref{cond.initpropag} this implies that $\sigma^i(t^-)=\sigma^i(t)$.

 Let $(a,t]$ be the largest interval on which $\sigma^i(s)=\sigma^i(t)$ for  $s\in (a,t]$. Then $a<t$ because  $\sigma^i(t^-)=\sigma^i(t)$. As $\sigma^i(\cdot)$ is discontinuous only when $X^i<0$, either $a=0$ or $X^i(a)<0$. Let $(b,t]$ be the largest interval contained in $(a,t]$ such that $X^{\sigma^i(s)}(s)-X^i(s) < \rho_{\max}^{-1}$. Note that either $b=a$ or $X^{\sigma^i(b)}(b)-X^i(b) = \rho_{\max}^{-1}$. On the other hand, $\dot X^i(s) = V(X^{\sigma^i(s)}(t)-X^i(s))=0$  and 
$$
 X^{\sigma^i(b)}(b)-X^i(b) \leq X^{\sigma^i(t)}(t)-X^i(t)  <  \rho_{\max}^{-1}.
 $$
Hence $b=a=0$, which leads to a contradiction with \eqref{hyp.condinti}.  
 \end{proof}

As a consequence, we have the following bound on $\rho^\ep_X$: 

\begin{Proposition} \label{prop.x+1-xpetit} We have
 $$
 \|\rho^{\ep, 0}_X\|_\infty+ \sup_{k=1,2}\|\rho^{\ep, k}_X\|_{L^\infty(\R\backslash (-\ep\rho_{\max}^{-1}, \ep\rho_{\max}^{-1}))} \leq \rho_{\max}, 
 $$
 while 
 $$
 \sup_{k=1,2}\|\rho^{\ep, k}_X\|_{L^1((-\ep\rho_{\max}^{-1}, \ep\rho_{\max}^{-1}))} \leq 3\ep. 
$$
\end{Proposition}

\begin{proof} 
Assume that $\rho_X^{\ep, k}(t,x)>\rho_{\max}$ for some $k\in \{0,1,2\}$ and $x\in \R$. Then,  by the definition of $\rho^{\ep,k}_X$,  there exists $i\in \{1, \dots, N-1\}$ such that $x\in [x^{\ep,i}(t), x^{\ep, \sigma^i(t)}(t))$, $X^{\sigma^i(\ep^{-1}t)}(\ep^{-1}t)-X^i(\ep^{-1}t)<\rho_{\max}^{-1}$. Moreover either $k=0$ and $X^i(\ep^{-1}t)\geq 0$, or $X^i(\ep^{-1}t)< 0$ and $r^i(\ep^{-1}t)=k\in \{1,2\}$.  Lemma \ref{lem.x+1-xpetit} then states that $X^i(\ep^{-1}t)<0$ and  $X^{\sigma^i(\ep^{-1}t)}(\ep^{-1}t)\geq 0$. Thus $k\in \{1,2\}$, $x^{\ep,i}(t)\geq -\ep \rho_{\max}^{-1}$ while $x^{\ep, \sigma^i(t)}(t)\leq \ep \rho_{\max}^{-1}$. This shows the first claim. 

For the second one, if $\rho^{\ep,k}_X\leq \rho_{\max}$ in $(-\ep\rho_{\max}^{-1}, \ep\rho_{\max}^{-1})$ the claim is obvious. Otherwise, using the notation above, 
\begin{align*}
\int_{(-\ep\rho_{\max}^{-1},\ep\rho_{\max}^{-1})}\rho^{\ep, k}_X & \leq 
\int_{(-\ep\rho_{\max}^{-1},\ep\rho_{\max}^{-1})\backslash [x^{\ep,i}(t), x^{\ep, \sigma^i(t)}(t))}\rho^{\ep, k}_X + 
\int_{[x^{\ep,i}(t), x^{\ep, \sigma^i(t)}(t))} y^{\ep, i}(t) \\
&\leq  2\ep + y^{\ep, i}(t)(x^{\ep, \sigma^i(t)}(t))-x^{\ep,i}(t))  \; =\;  3 \ep. 
\end{align*}
\end{proof}

\section{$L^1-$type estimates on a branch}\label{sec.Esti1branch}

We discuss here the classical relationship between a follow-the-leader model of the form 
$$
\dot X^i(t)= V(X^{i+1}(t)-X^i(t))
$$
and the conservation law 
\be\label{eq.CL}
\rho_t + (f(\rho))_x=0
\ee
on a single line or half-line. Throughout this section, we assume for simplicity that $X^{i+1}\geq X^i$, which is possible since the equation on a line or a half line preserves the order. Let us recall that $f$ and $V$ are linked by the relation $f(r)= rV(1/r)$ and satisfy Assumption {\bf (H)} of Subsection \ref{subsec.hyp}. We consider two cases: either the flow of the $(X^i)$ is unconstrained and $\dot X^N(t)=V_{\max}$. Or the flow is constrained to stay in the half-line $\R_-$ by a stop at $x=0$. In this case $\dot X^N(t)= V(-X^N(t))$. The first case is classical and has been studied thoroughly: we recall below the main results of \cite{rigorousLWR}.  The stopped case is handled by  reducing the problem to the previous one. 

\subsection{BV estimate and Kato's inequality  on a line} 

Let $(X^i)_{i=1, \dots, N}$ be a solution to the equation 
$$
\dot X^i(t)= V(X^{i+1}(t)-X^i(t)), \qquad t\geq 0, \; i\in \{1, \dots, N\}.
$$
By convention, $X^{N+1}\equiv \infty$ and thus $\dot X^N(t)=V_{\max}$. We assume that $X^i(0)< X^{i+1}(0)$ for any $i$. This order is preserved along the flow. We fix $\ep>0$ and set  
\be\label{def.rhoXDiFR}
\rho_X(t,x) = \sum_{i=1}^{N-1} \frac{1}{X^{i+1}(\ep^{-1}t)-X^i(\ep^{-1} t)} {\bf 1}_{[\ep X^i(\ep^{-1}t), \ep X^{i+1}(\ep^{-1} t))}.
\ee
The central estimate used in this paper comes from Di~Francesco-Rosini  \cite{rigorousLWR}:   

\begin{Proposition}[\cite{rigorousLWR}]\label{prop.DiFR} Under our standing assumptions, we have: 
\begin{itemize}
\item  (BV estimate) There is a constant $C$, depending on the diameter of the support of $\rho_X(0, \cdot)$, such that, for any $t\geq0$, 
$$
TV(v(\rho_X(t,\cdot))) \leq \min\{TV(v(\rho_X(0,\cdot))), C(1+t^{-1})\}. 
$$

\item (Approximated entropy solution) For any $k\geq 0$ and any $\phi\in C^1_c((0,\infty)\times \R)$ with $\phi\geq 0$, 
\be\label{ineq.K}
\int_0^\infty \int_\R |\rho_X-k| \phi_t + {\rm sgn}(\rho_X-k) (f(\rho_X)-f(k)) \phi_x   \geq - \ep \sup_t TV(v(\rho_X(t))\int_0^\infty \|\phi_x(t, \cdot)\|_\infty dt.
\ee
Here $\sup_t TV(v(\rho_X(t))$ denotes the total variation of $v(\rho(t,\cdot))$ in the support of the test function $\phi$. 
\end{itemize}
\end{Proposition} 

Note that \eqref{ineq.K} means that $\rho_X$ is an approximate entropy solution of the conservation law \eqref{eq.CL}. 

\begin{proof} We  check below the proof of the inequality
\be\label{ljksdfjn,qsdlkf}
TV(v(\rho_X(t,\cdot))) \leq TV(v(\rho_X(0,\cdot))), 
\ee
as it is not given explicitly in \cite{rigorousLWR} (however the argument is exactly that of \cite[Proposition 3.5]{rigorousLWR}). The inequality  
$$
TV(v(\rho_X(t,\cdot)))\leq C(1+t^{-1}) 
$$
is given in \cite[Proposition 3.6]{rigorousLWR}, while the fact that $\rho_X$ is an approximated entropy solution is the key argument of the proof of \cite[Theorem 2.3]{rigorousLWR}. 

To prove \eqref{ljksdfjn,qsdlkf} we first note that 
$$
TV(v(\rho_X(t,\cdot)))=  \sum_{i=0}^{N-1} \left|v(y^{i+1}(t))-v(y^i(t))\right|, 
$$
where, for $i=1, \dots, N-1$, 
 $$
 y^i(t):= \frac{1}{X^{i+1}(\ep^{-1}t)-X^i(\ep^{-1}t)}
 $$
 and where, by convention, $X^0(t)=-\infty$, $X^{N+1}(t)=\infty$ and thus $y^0(t)=y^{N}(t)=0$. 
 
We note that, for $i=1, \dots, N-1$, 
 $$
 \dot y^i(t)= - (y^i(t))^2 \ep^{-1} \left( v(y^{i+1}(t))-v(y^{i}(t))\right). 
 $$
 Thus 
 \begin{align*} 
\frac{d}{dt} TV(v(\rho_X(t,\cdot))) &=   \sum_{i=0}^{N-1}  {\rm sgn} \left( v( y^{i+1}(t))-v(y^i(t))\right) 
\left(v'(y^{i+1}(t))\dot y^{i+1}(t)-v'(y^i(t))\dot y^{i}(t)\right) \\ 
&= - {\rm sgn} \left( v(y^{1}(t))-v(y^0(t))\right) v'(y^0(t)) \dot y^{0}(t)+  {\rm sgn} \left( v(y^{N}(t))-v(y^{N-1}(t))\right) v'(y^N(t)) \dot y^{N}(t) \\
& \qquad + \sum_{i=1}^{N-1}v'(y^i(t)) \dot y^{i}(t) \left( {\rm sgn} \left( v(y^{i}(t))-v(y^{i-1}(t))\right) - {\rm sgn} \left(  v(y^{i+1}(t))- v(y^{i}(t))\right)\right).
\end{align*}
The first two terms vanish. 
As, for $i=1, \dots, N-1$,  
\begin{align*}
&v'(y^i(t))\dot y^{i}(t) \left( {\rm sgn} \left( y^{i}(t)-y^{i-1}(t)\right) - {\rm sgn} \left(  y^{i+1}(t)- y^{i}(t)\right)\right) \\ 
& = -  v'(y^i(t))(y^i(t))^2 \ep^{-1}  \left( v(y^{i+1}(t))-v(y^{i}(t))\right) \left( {\rm sgn} \left( v(y^{i}(t))-v(y^{i-1}(t))\right) - {\rm sgn} \left( v(y^{i+1}(t))- v(y^{i}(t))\right)\right) \leq0, 
\end{align*}
since $z\to v(z)$ is non increasing, we infer that 
$$
\frac{d}{dt} TV(\rho_X(t,\cdot))) \leq 0.
$$
\end{proof} 

We now consider $\rho_Y$ such that either $\rho_Y$ is a bounded entropy solution of \eqref{eq.CL}, or 
\be\label{def.rhoYYY}
\rho_Y(t,x) = \sum_{i=1}^{M-1} \frac{1}{Y^{i+1}(\ep^{-1}t)-Y^i(\ep^{-1}t)} {\bf 1}_{[\ep Y^i(\ep^{-1}t), \ep Y^{i+1}(\ep^{-1}t))}, 
\ee
where  $(Y^j)_{i=1, \dots, M}$ is another solution to the follow-the-leader model:
$$
\dot Y^i(t)= V(Y^{i+1}(t)-Y^i(t)), \qquad t\geq 0, \; i\in \{1, \dots, M\},
$$
with again the convention that $Y^{M+1}(t)=\infty$. 

From now on we fix $\delta:\R\to \R_+$ a smooth function, with support in $[-1,1]$, and integral $1$: $\int_\R \delta =1$. We set $\delta_h(x)= h\delta (hx)$ for $h>0$ large. 

\begin{Corollary}[Kato's inequality]\label{cor.estiL1} For any $0\leq t_1\leq t_2\leq 1$, any $\psi\in C_c^\infty([0,\infty)\times \R, \R_+)$ and $h\geq 1$, 
\begin{align*}
&\iint \Bigl\{  |\rho_X(t_2,x)-\rho_Y(t_2,y)|  \psi(t_2, \frac{x+y}{2}) \delta_h(\frac{y-x}{2}) \Bigr\}\; dxdy\\ 
&+ \iint\int_{t_1}^{t_2} \Bigl\{  |\rho_X(t,x)-\rho_Y(t,y)|  \psi_t(t, \frac{x+y}{2}) \delta_h(\frac{y-x}{2}) \\ 
& \qquad + {\rm sgn}(\rho_X(t,x)-\rho_Y(t,y)) (f(\rho_X(t,x))-f(\rho_Y(t,y)))\psi_x(t, \frac{x+y}{2}) \delta_h(\frac{y-x}{2})
\Bigr\}\; dtdxdy \\ 
& \geq \iint \Bigl\{  |\rho_X(t_1,x)-\rho_Y(t_1,y)|  \psi(t_1, \frac{x+y}{2}) \delta_h(\frac{y-x}{2}) \Bigr\}\; dxdy
- C\ep K  \int_{t_1}^{t_2} (h^2 \|\psi(t, \cdot)\|_\infty+ h\|\psi_x(t, \cdot)\|_\infty)dt.
\end{align*}
where $C$ depends on the support of $\psi$ (but not on $\ep$ and $h$) and
\be\label{defKKK}
K=\left\{\begin{array}{ll} \sup_t TV(v(\rho_X(t))) +\sup_t TV(v(\rho_Y(t))) & \text{ if $\rho_Y$ is given by \eqref{def.rhoYYY}} ,\\
\sup_t TV(v(\rho_X(t))) & \text{if $\rho_Y$ is a solution of \eqref{eq.CL},}
\end{array}\right. 
\ee
(the total variation being computed on the support of $\psi$). 
\end{Corollary}


\begin{proof} Let $\psi\in C_c^\infty((0,\infty)\times \R, \R_+)$. 
Following Kruzkhov \cite{K70} we apply inequality \eqref{ineq.K} with $k=\rho_Y(s,y)$ and  
$$
\phi(t,x,s,y)= \delta_n(\frac{s-t}{2})\delta_h(\frac{y-x}{2}) \psi(\frac{t+s}{2}, \frac{x+y}{2}) ,
$$
where $n,h\geq 1$.  We obtain 
\begin{align*}
&\int_0^\infty \int_\R\Bigl\{  |\rho_X(t,x)-\rho_Y(s,y)| \left( -\delta'_n(\frac{s-t}{2}) \psi(\frac{t+s}{2}, \frac{x+y}{2}) + \delta_n(\frac{s-t}{2})\psi_t(\frac{t+s}{2}, \frac{x+y}{2})\right) \delta_h(\frac{y-x}{2}) \\ 
& + {\rm sgn}(\rho_X(t,x)-\rho_Y(s,y)) (f(\rho_X(t,x))-f(\rho_Y(s,y)))\times \\
& \qquad \times  \left( - \delta_h'(\frac{y-x}{2}) \psi(\frac{t+s}{2}, \frac{x+y}{2})  + \delta_h(\frac{y-x}{2})\psi_x(\frac{t+s}{2}, \frac{x+y}{2})\right)\delta_n(\frac{s-t}{2})
\Bigr\}\; dxdt \\ 
& \geq - \ep \sup_s TV(v(\rho_X(s))) \int_0^\infty (\|\delta_h'\|_\infty\|\psi(\frac{t+s}{2}, \cdot)\|_\infty+ \|\delta_h\|_\infty\|\psi_x(\frac{t+s}{2}, \cdot)\|_\infty)\delta_n(\frac{s-t}{2})dt. 
\end{align*}
We integrate in $(s,y)$ (over the support of $\psi$) and add the symmetric inequality for $\rho_Y$ to get 
\begin{align*}
&\iiiint \Bigl\{  |\rho_X(t,x)-\rho_Y(s,y)|  \delta_n(\frac{s-t}{2})\psi_t(\frac{t+s}{2}, \frac{x+y}{2}) \delta_h(\frac{y-x}{2}) \\ 
& + {\rm sgn}(\rho_X(t,x)-\rho_Y(s,y)) (f(\rho_X(t,x))-f(\rho_Y(s,y))) \delta_h(\frac{y-x}{2})\psi_x(\frac{t+s}{2}, \frac{x+y}{2})\delta_n(\frac{s-t}{2})
\Bigr\}\; dtdxdsdy \\ 
& \geq - C\ep K \iint (\|\delta_h'\|_\infty\|\psi(\frac{t+s}{2}, \cdot)\|_\infty+ \|\delta_h\|_\infty\|\psi_x(\frac{t+s}{2}, \cdot)\|_\infty)\delta_n(\frac{s-t}{2})dsdt ,
\end{align*}
where $K$ is defined by \eqref{defKKK} and $C$ depends on the support of $\psi$. 

We let $n\to\infty$ to obtain:
\begin{align*}
&\iiint \Bigl\{  |\rho_X(t,x)-\rho_Y(t,y)|  \psi_t(t, \frac{x+y}{2}) \delta_h(\frac{y-x}{2}) \\ 
& + {\rm sgn}(\rho_X(t,x)-\rho_Y(t,y)) (f(\rho_X(t,x))-f(\rho_Y(t,y))) \delta_h(\frac{y-x}{2})\psi_x(t, \frac{x+y}{2})
\Bigr\}\; dtdxdy \\ 
& \geq - C\ep K  \int_0^\infty (h^2 \|\psi(t, \cdot)\|_\infty+ h\|\psi_x(t, \cdot)\|_\infty)dt.
\end{align*}
We complete the proof by a standard cut-off in time argument (we use here the weak continuity in time of the solution, which holds thanks to \cite[Proposition 3.8]{rigorousLWRbis}). 
\end{proof}

\subsection{BV  estimates and Kato's inequality on a road with a red light} \label{subsec.stoppedtraf}

We  now consider a traffic flow on a road $(-\infty, 0)$ stopped  by a red light at  $x=0$. More precisely, we assume that the $(X^i)_{i=1, \dots, N}$ are strictly ordered at initial time: $X^i(0)<X^{i+1}(0)<0$ for any $i=1, \dots, N-1$,  and solve the equation
$$
\dot X^i(t)= \left\{ \begin{array}{ll}
V( X^{i+1}(t)-X^i(t)) & \text{for}\; i\in \{1, \dots, N-1\},\\ 
V(-X^{N}(t))  & \text{for}\; i=N.
\end{array}\right.
$$
It is thus natural to set $X^{N+1}(t)\equiv 0$ for any $t$, where $X^{N+1}$ represents a red light preventing the cars to leave the road $(-\infty, 0)$. We fix a scale $\ep>0$ and, as before, we set 
$$
x^i(t)= \ep X^i(\ep^{-1}t), \qquad y^i(t)= \frac{1}{X^{i+1}(\ep^{-1}t)-X^i(\ep^{-1}t)}, 
$$
\be\label{def.rhoXStopped}
\rho_X(t,x) = \sum_{i=1}^{N-1} y^i(t) {\bf 1}_{[x^i(t), x^{i+1}(t))}. 
\ee
Note that, by convention, $\rho_X$ vanishes on $(-\infty, x^1(t))$ and on $(x^{N}(t),\infty)$. Let us first start with a BV bound. 

\begin{Proposition}\label{prop.BoundBVRed} For any $t\geq 0$,  
$$
TV(v(\rho_X(t,\cdot))) \leq TV(v(\rho_X(0,\cdot))) +4V_{\max},
$$
where $V_{\max} =\|V\|_\infty$. 
\end{Proposition}

\begin{proof}  We consider the solution $(\tilde X^i)_{i\in \N}$  to the equation 
$$
\dot{\tilde X}^j(t)= V(\tilde X^{j+1}(t)-\tilde X^j(t)), \qquad t\geq 0, \; i\in \N, 
$$
with 
$$
\tilde X^i(0)= X^i(0) \; \text{for} \; i\in \{1, \dots, N\}, 
$$
and 
$$
 \tilde X^i(0)= 
 (i-N-1) \rho_{\max}^{-1} \qquad  \text{for}\; i\geq N+1. 
$$
We note that 
$$
\tilde X^i(t)= X^i(t) \; \text{if} \; i\in \{1, \dots, N\}, \qquad \tilde X^i(t)= \tilde X^i(0)\; \text{if}\; i\geq N+1.
$$
We set 
$$
\tilde \rho_X(t,x) = \sum_{i\in \N} \tilde y^i(t){\bf 1}_{[\tilde x^i(t), \tilde x^{i+1}(t))},
$$
where 
$$
\tilde y^i(t)= \frac{1}{\tilde X^{i+1}(\ep^{-1}t)-\tilde X^i(\ep^{-1}t)}, \qquad \tilde x^i(t)=\ep \tilde X^i(\ep^{-1}t).
$$
We note that 
$$
\tilde \rho_X(t,x)= \rho_{\max} \qquad \text{for}\; x\geq 0, \qquad \tilde \rho_X(t,x)= \rho_X(t,x)\; \text{for}\; x\leq x^N(t). 
$$
By Proposition \ref{prop.DiFR} (see also its proof) 
$$
TV(v(\tilde \rho(t,\cdot))) \leq TV(v(\tilde \rho(0,\cdot)))  \qquad \forall t\geq 0.
$$
As 
$$
 \tilde \rho(t,x)= \rho(t,x) -\frac{\ep}{x^N(t)}{\bf 1}_{[x^N(t), 0)} + \rho_{\max}{\bf 1}_{[0,\infty)},
 $$
 we have (recalling that $v(0)= V_{\max}$ and $v(\rho_{\max})=0$)
$$
v( \tilde \rho(t,x))= v(\rho(t,x))+v( \frac{-\ep}{x^N(t)}){\bf 1}_{[x^N(t), 0)} - V_{\max}{\bf 1}_{[x^N(t), \infty)}. 
 $$
 Thus 
 $$
|TV(v(\tilde \rho(t,\cdot))-TV(v( \rho(t,\cdot))|\leq 2V_{\max},
$$
which proves the result. 
\end{proof}

We now turn to the (approximate) Kato's inequality. We consider two cases: either $Y=(Y^i)_{i=1, \dots, M}$ is  a solution to a follow-the-leader model  with a red light at $x=0$: 
\be\label{eq.Yifeu}
\dot Y^i(t)= \left\{ \begin{array}{ll}
V( Y^{i+1}(t)-Y^i(t)) & \text{for}\; i\in \{1, \dots, M-1\},\\ 
V(-Y^{N}(t))  & \text{for}\; i=M,
\end{array}\right.
\ee
and 
\be\label{def.rohYYYstopped}
\rho_Y(t,x) = \sum_{i=1}^{N-1} \frac{1}{Y^{i+1}(\ep^{-1}t)-Y^i(\ep^{-1}t)} {\bf 1}_{[\ep Y^i(\ep^{-1}t), \ep Y^{i+1}(\ep^{-1} t))}. 
\ee
Or $\rho_Y$ is an entropy solution to 
\be\label{eq.CLboundary}
\begin{array}{l}
\rho_Y(t,x) \in [0,\rho_{\max}] \qquad \text{a.e. in}\; (0,\infty)\times (-\infty,0)\\
(\rho_Y)_t+(f(\rho_Y))_x= 0 \qquad \text{in}\; (0,\infty)\times (-\infty,0)\\ 
f(\rho_Y(t,0^-))= 0 \qquad \text{in}\; (0,\infty).
\end{array} 
\ee
By an entropy solution, we mean that $\rho_Y$ is an entropy solution of the conservation law in $(0,\infty)\times (-\infty,0)$ and its trace at $x=0^-$ satisfies $f(\rho_Y(t,0^-))= 0$ for a.e. $t\in (0,\infty)$. Recall that this strong trace exists thanks to \cite{Ps07}. 

As before we fix $\delta:\R\to \R_+$ a smooth function, with support in $[-1,1]$, and integral $1$: $\int_\R \delta =1$. We set $\delta_h(x)= h\delta (hx)$ for $h>0$ large.

\begin{Proposition}\label{prop.rateFeu} For any $0\leq t_1\leq t_2$ and any  test function $\psi\in C^\infty_c([0,\infty)\times \R,\R_+)$ and $h\geq 1$, 
\begin{align*}
&\iint \Bigl\{  |\rho_X(t_2,x)-\rho_Y(t_2,y)|  \psi(t_2, \frac{x+y}{2}) \delta_h(\frac{y-x}{2})\Bigr\}dxdy \\ 
&+ \iint \int_{t_1}^{t_2}  \Bigl\{  |\rho_X(t,x)-\rho_Y(t,y)|  \psi_t(t, \frac{x+y}{2}) \delta_h(\frac{y-x}{2}) \\ 
& \qquad + {\rm sgn}(\rho_X(t,x)-\rho_Y(t,y)) (f(\rho_X(t,x))-f(\rho_Y(t,y))) \psi_x(t, \frac{x+y}{2}) \delta_h(\frac{y-x}{2})
\Bigr\}\; dtdxdy \\ 
& \geq \iint \Bigl\{  |\rho_X(t_1,x)-\rho_Y(t_1,y)|  \psi(t_1, \frac{x+y}{2}) \delta_h(\frac{y-x}{2})\Bigr\}dxdy \\ 
& \qquad - C \int_{t_1}^{t_2} (\ep K h^2 \|\psi(t, \cdot)\|_\infty+ (\ep Kh+h^{-1})\|\psi_x(t, \cdot)\|_\infty +
 (\ep +h^{-1})\|\psi_t(t, \cdot)\|_\infty  )dt,
\end{align*}
where $C$ depends on the support of $\psi$ (but not on $\ep$ and $h$) and 
$$
K= \left\{\begin{array}{ll}
\sup_t TV( v( \rho_X(t))) +\sup_t TV( v( \rho_Y(t))) +1 & \mbox{\rm if $\rho_Y$ is given by \eqref{def.rohYYYstopped},}\\
\sup_t TV( v( \rho_X(t)))  +1& \mbox{\rm if $\rho_Y$ solves \eqref{eq.CLboundary}.}
\end{array}\right.
$$
\end{Proposition}


To prove the proposition in the case where $\rho_Y$ solves \eqref{eq.CLboundary}, we need the following remark.

\begin{Lemma}\label{lem.tilderho} If $\rho$ is an entropy solution to \eqref{eq.CLboundary}, then 
$$
\tilde \rho(t,x):= \left\{\begin{array}{ll}
\rho(t,x) & \text{in}\;  (0,\infty)\times (-\infty,0)\\ 
\rho_{\max} &  (0,\infty)\times (0,\infty)
\end{array}\right. 
$$
is an entropy solution to 
$$
\begin{array}{l}
\tilde \rho(t,x) \in [0,\rho_{\max}] \qquad \text{a.e. in}\; (0,\infty)\times \R\\
(\tilde \rho)_t+(f(\tilde  \rho))_x= 0 \qquad \text{in}\; (0,\infty)\times (-\infty,0)\\ 
\end{array}
$$
\end{Lemma}

\begin{proof} Let $\phi$ be a nonnegative smooth test function with compact support in $(0,\infty)\times \R$. We have to check that, for any $k\geq0$,  
$$
\int_0^\infty\int_{\R} | \tilde \rho-k|\phi_t(t,x) + {\rm sgn}(\tilde \rho-k)(f(\tilde \rho)-f(k)) \phi_x (t,x) \geq 0. 
$$
Let $\psi:\R\to [0,1]$ be a smooth and even function such that $\psi(x)=1$ for $|x|\geq 1$, $\psi(x)=0$ for $|x|\leq 1/2$, $\psi$ is nondecreasing on $(0,\infty)$. For $\ep>0$ we set $\psi^\ep(x)= \psi(\ep^{-1}x)$. As $\tilde \rho$ is an entropy solution on $(0,\infty)\times (-\infty,0)$ and on $(0,\infty)\times (0,\infty)$, we have (for the test function $(t,x)\to \psi^\ep(x)\phi(t,x)$): 
\be\label{aiqkshldjnxcv}
0 \leq \int_0^\infty\int_\R  | \tilde \rho-k|\psi^\ep \phi_t+ {\rm sgn}(\tilde \rho-k)(f(\tilde \rho)-f(k)) (\psi^\ep_x \phi+\psi^\ep \phi_x ) .
\ee
As $\psi^\ep \to1$ a.e., one easily checks that 
$$
\lim_{\ep\to 0^+} \int_0^\infty\int_\R  | \tilde \rho-k|\psi^\ep \phi_t = \int_0^\infty\int_\R | \tilde \rho-k| \phi_t 
$$
and
$$
\lim_{\ep\to 0^+} \int_0^\infty\int_\R {\rm sgn}(\tilde \rho-k)(f(\tilde \rho)-f(k)) \psi^\ep \phi_x  =  \int_0^\infty\int_\R{\rm sgn}(\tilde \rho-k)(f(\tilde \rho)-f(k))\phi_x  .
$$
On the other hand 
\begin{align*}
& \int_0^\infty\int_\R  {\rm sgn}(\tilde \rho-k)(f(\tilde \rho)-f(k)) \psi^\ep_x \phi = \int_0^\infty\int_{[-\ep,\ep]}  {\rm sgn}(\tilde \rho(t,x)-k)(f(\tilde \rho)-f(k)) \psi^\ep_x \phi\\ 
&\qquad = \int_0^\infty\int_{[-1,1]}  {\rm sgn}(\tilde \rho(t,\ep y)-k)(f(\tilde \rho(t,\ep y))-f(k)) \psi_x( y) \phi(t,\ep y). 
\end{align*}
Hence, by definition of $\tilde \rho$ and the boundary condition  satisfied by $\tilde \rho$, 
\begin{align*}
& \lim_{\ep \to 0^+} \int_0^\infty\int_\R  {\rm sgn}(f(\tilde \rho)-f(k)) \psi^\ep_x \phi\\
&\qquad  =  \int_0^\infty(\int_{[-1,0]}  {\rm sgn}( \rho(t,0^-)-k)(f( \rho(t,0^-))-f(k)) \psi_x( y) \phi(t,0)\\
& \qquad \qquad + \int_{[0,1]}  {\rm sgn}(\rho_{\max}-k)(f(\rho_{\max})-f(k)) \psi_x(y) \phi(t,0) \; )\\ 
&\qquad = \int_0^\infty(\int_{[-1,0]} - {\rm sgn}( \rho(t,0^+)-k) f(k) \psi_x( y) \phi(t,0) - \int_{[0,1]}  f(k) \psi_x(y) \phi(t,0) \; ) \\ 
&\qquad = \int_0^\infty(\int_{[-1,0]} ( f(k)  \phi(t,0) ({\rm sgn}(\rho(t,0^+)-k) -1)  \; )  \; \leq \; 0,
\end{align*}
where we used the fact that $\psi(0)=0$ and $\psi(\pm 1)=1$ in the last equality. Thus letting $\ep\to 0$ in \eqref{aiqkshldjnxcv} gives the result. 
\end{proof}

\begin{proof}[Proof of Proposition \ref{prop.rateFeu}] We start with a  construction similar to the one in the proof of Proposition \ref{prop.BoundBVRed}: we consider the solution $(\tilde X^i)_{i\in \N}$  to the equation 
$$
\dot{\tilde X}^j(t)= V(\tilde X^{j+1}(t)-\tilde X^j(t)), \qquad t\geq 0, \; i\in \N, 
$$
with 
$$
\tilde X^i(0)= X^i(0) \; \text{for} \; i\in \{1, \dots, N\}, 
$$
and 
$$
 \tilde X^i(0)= \left\{\begin{array}{ll}
 (i-N-1) \rho_{\max}^{-1}  & \text{if}\; X^N(0)\leq -\rho_{\max}^{-1} \\ 
 X^N(0) +  (i-N) \rho_{\max}^{-1}&\text{otherwise}
 \end{array}\right. \qquad  \text{for}\; i\geq N+1. 
$$
We note that 
$$
\tilde X^i(t)= X^i(t) \; \text{if} \; i\in \{1, \dots, N\}, \qquad \tilde X^i(t)= \tilde X^i(0)\; \text{if}\; i\geq N+1.
$$
We set 
$$
\tilde \rho_X(t,x) = \sum_{i\in \N} \tilde y^i(t){\bf 1}_{[\tilde x^i(t), \tilde x^{i+1}(t))},
$$
where 
$$
\tilde y^i(t)= \frac{1}{\tilde X^{i+1}(\ep^{-1}t)-\tilde X^i(\ep^{-1}t)}, \qquad \tilde x^i(t)=\ep \tilde X^i(\ep^{-1}t).
$$
We note that 
$$
\tilde \rho_X(t,x)= \rho_{\max} \qquad \text{for}\; x\geq 0, \qquad \tilde \rho_X(t,x)= \rho_X(t,x)\; \text{for}\; x\leq x^N(t). 
$$
If $Y$ solves \eqref{eq.Yifeu}, we  define $\tilde \rho_Y$ in a symmetric way. In the case where $\rho_Y$ solves \eqref{eq.CLboundary}, we use  Lemma \ref{lem.tilderho} to define $\tilde \rho_Y$. We then obtain by Corollary  \ref{cor.estiL1}, as $\tilde \rho_X=\tilde \rho_Y= \rho_{\max}$ in $[0,\infty)\times [0,\infty)$: for any $\psi\in C_c^\infty((0,\infty)\times \R, \R_+)$ and $h\geq 1$, 
\begin{align*}
&\iiint \Bigl\{  |\tilde \rho_X(t,x)-\tilde \rho_Y(t,y)|  \psi_t(t, \frac{x+y}{2}) \delta_h(\frac{y-x}{2}) \\ 
& \qquad + {\rm sgn}(\tilde \rho_X(t,x)-\tilde \rho_Y(t,y)) (f(\tilde \rho_X(t,x))-f(\tilde \rho_Y(t,y))) \delta_h(\frac{y-x}{2})\psi_x(t, \frac{x+y}{2})
\Bigr\}\; dtdxdy \\ 
& \geq - C\ep \tilde K  \int_0^\infty (h^2 \|\psi(t, \cdot)\|_\infty+ h\|\psi_x(t, \cdot)\|_\infty)dt ,
\end{align*}
where 
$$
\tilde K=\left\{\begin{array}{ll}  \sup_t TV(v(\tilde \rho_X(t))) +\sup_t TV(v (\tilde \rho_Y(t))) & \mbox{\rm if $\rho_Y$ solves \eqref{def.rohYYYstopped},}\\
 \sup_t TV(v(\tilde \rho_X(t))) & \mbox{\rm if $\rho_Y$ solves \eqref{eq.CLboundary}.}
\end{array}\right. 
$$
As $\tilde \rho_X=\tilde \rho_Y$ in $[0,\infty)\times [0,\infty)$ and  $\delta$ has a support in $[-1,1]$, we have 
\begin{align*}
& \iiint \Bigl\{  |\tilde \rho_X(t,x)-\tilde \rho_Y(t,y)|  \psi_t(t, \frac{x+y}{2}) \delta_h(\frac{y-x}{2})\Bigr\} dxdydt \\
&= \int_0^\infty \int_{-\infty}^{h^{-1}}\int_{-\infty}^{h^{-1}} \Bigl\{  |\tilde \rho_X(t,x)-\tilde \rho_Y(t,y)|  \psi_t(t, \frac{x+y}{2}) \delta_h(\frac{y-x}{2})\Bigr\}  dxdydt .
\end{align*}
The same equality holds for $\rho_X$ and $\rho_Y$. 
Hence 
\begin{align*}
& \Bigl |  \iiint \Bigl\{  |\tilde \rho_X(t,x)-\tilde \rho_Y(t,y)| \psi_t(t, \frac{x+y}{2}) \delta_h(\frac{y-x}{2})\Bigr\} dxdydt\\
& \qquad-  \iiint \Bigl\{  | \rho_X(t,x)- \rho_Y(t,y)|  \psi_t(t, \frac{x+y}{2})  \delta_h(\frac{y-x}{2})\Bigr\} dxdydt \Bigr | \\
& \leq    \int_0^\infty\iint_{(-\infty,h^{-1})^2} \Bigl\{  |\tilde \rho_X(t,x)- \rho_X(t,y)| | \psi_t(t, \frac{x+y}{2})| \delta_h(\frac{y-x}{2})\Bigr\} dxdydt\\
& \qquad +    \int_0^\infty\iint_{(-\infty,h^{-1})^2} \Bigl\{  |\tilde \rho_Y(t,x)- \rho_Y(t,y)| | \psi_t(t, \frac{x+y}{2})| \delta_h(\frac{y-x}{2})\Bigr\} dxdydt.
\end{align*}
On the other hand, recalling that $\tilde \rho_X(t,\cdot)= \rho_X(t,\cdot)$  in $(-\infty, x^N(t))$ and that $\delta_h$ has integral $1$, 
\begin{align*}
& \int_{-\infty}^{h^{-1}}\int_{-\infty}^{h^{-1}}   |\tilde \rho_X(t,x) -\rho_X(t,x)| | \psi_t(t, \frac{x+y}{2})|  \delta_h(\frac{y-x}{2})\ dxdy \\ 
&\qquad  \leq \|\psi_t(t,\cdot)\|_\infty
\int_{x^N(t)}^{h^{-1}} |\tilde \rho_X(t,x) -\rho_X(t,x)|  dx . 
\end{align*}
If $X^N(0)\leq -\rho_{\max}^{-1}$, then $\tilde x^{N+1}(t)=0$ and
\begin{align*}
& \int_{x^N(t)}^{h^{-1}} |\tilde \rho_X(t,x) -\rho_X(t,x)|  dx \\
& \qquad = \int_{x^N(t)}^{\tilde x^{N+1}(t)} |\tilde \rho_X(\cdot,x) -\rho_X(t,x)|  dx + \int_{\tilde x^{N+1}(t)}^{h^{-1}} |\tilde \rho_X(\cdot,x) -\rho_X(t,x)| dx \\ 
& \qquad \leq  ({\tilde x^{N+1}(t)}-x^N(t))\tilde y^N(t)  + \rho_{\max} h^{-1}  \; \leq  \; \ep +\rho_{\max} h^{-1}. 
\end{align*}
If $X^N(0)> -\rho_{\max}^{-1}$, then 
\begin{align*}
& \int_{x^N(t)}^{h^{-1}} |\tilde \rho_X(t,x) -\rho_X(t,x)|  dx 
 = (h^{-1}-x^N(t))\rho_{\max}    \; \leq \; \rho_{\max} (\rho_{\max}\ep +h^{-1}). 
\end{align*}
Arguing in the same way for $\rho_Y$, this shows that 
\begin{align*}
& \Bigl |  \iiint \Bigl\{  |\tilde \rho_X(t,x)-\tilde \rho_Y(t,y)| \psi_t(t, \frac{x+y}{2}) \delta_h(\frac{y-x}{2})\Bigr\} dt dxdy\\
& \qquad-  \iiint \Bigl\{  | \rho_X(t,x)- \rho_Y(t,y)|  \psi_t(t, \frac{x+y}{2})  \delta_h(\frac{y-x}{2})\Bigr\} dt dxdy \Bigr |  \leq C (\ep +h^{-1}) \int_0^\infty \|\psi_t(t,\cdot)\|_\infty dt. 
\end{align*}
In the same way, 
\begin{align*}
&\Bigl | \iiint \Bigl\{ {\rm sgn}(\tilde \rho_X(t,x)-\tilde \rho_Y(t,y)) (f(\tilde \rho_X(t,x))-f(\tilde \rho_Y(t,y))) \delta_h(\frac{y-x}{2})\psi_x(t, \frac{x+y}{2}) \Bigr\}\; dtdxdy \\ 
& \qquad - \iiint \Bigl\{ {\rm sgn}( \rho_X(t,x)- \rho_Y(t,y)) (f( \rho_X(t,x))-f( \rho_Y(t,y))) \delta_h(\frac{y-x}{2})\psi_x(t, \frac{x+y}{2}) \Bigr\}\; dtdxdy\Bigr | \\ 
& \leq C(\ep +h^{-1}) \int_0^\infty \|\psi_x(t,\cdot)\|_\infty  dt . 
\end{align*}
To conclude, we  need to estimate $\tilde K$. We have 
\begin{align*}
TV(v(\tilde \rho_X(t))) & = \sum_{i=0}^\infty |v(\tilde y^{i+1}(t))-v(\tilde y^i(t))| = \sum_{i=0}^{N-2} |v( y^{i+1}(t))-v( y^i(t))|+ |-v(\tilde y^{N-1}(t))|\\  
& \leq TV(v( \rho_X(t)))+ \|v\|_\infty \leq CK, 
\end{align*} 
since $v(\tilde y^i(t))= v(\rho_{\max})= 0$ for $i\geq N$. The symmetric inequality holds for $\rho_Y$. This proves that 
\begin{align*}
&\iiint \Bigl\{  |\rho_X(t,x)-\rho_Y(t,y)|  \psi_t(t, \frac{x+y}{2}) \delta_h(\frac{y-x}{2}) \\ 
& \qquad + {\rm sgn}(\rho_X(t,x)-\rho_Y(t,y)) (f(\rho_X(t,x))-f(\rho_Y(t,y)))\psi_x(t, \frac{x+y}{2}) \delta_h(\frac{y-x}{2})
\Bigr\}\; dtdxdy \\ 
& \geq - C \int_0^\infty (\ep K h^2 \|\psi(t, \cdot)\|_\infty+ (\ep Kh+h^{-1})\|\psi_x(t, \cdot)\|_\infty +
 (\ep +h^{-1})\|\psi_t(t, \cdot)\|_\infty  )dt,
\end{align*}
We conclude then by a cut-off argument. 
\end{proof}

\subsection{A local  compactness  result in $L^1$}

In order to pass to the limit in our problem, we need to establish some $L^1$ compactness of the solution far from the junction. More precisely, we consider a family of solutions $(X^{\ep, i})_{i=1, \dots, N}$ to the follow-the-leader system: 
$$
\dot X^{\ep,N}(t)\geq0, \qquad \dot X^{\ep, i}(t)= V(X^{\ep, i+1}(t)-X^{\ep,i}(t)), \qquad i=1, \dots, N-1, \; t\geq 0. 
$$
Note that, in contrast with the previous subsection, we do not require any special behavior of $X^{\ep,N}$. We set as usual 
$$
\rho^\ep(t,x) = \sum_{i=1}^{N-1} y^{\ep, i}(t) {\bf 1}_{[x^{\ep, i}(t), x^{\ep, i+1}(t))}
$$
where 
$$
x^{\ep, i}(t)= \ep X^{\ep, i}(\ep^{-1}t) , \qquad i=1, \dots, N,
$$
and 
$$
y^{\ep,i}(t)= \frac{1}{X^{\ep, i+1}(t)-X^{\ep, i}(t)} = \frac{\ep}{x^{\ep,i+1}(t)-x^{\ep,i}(t)}, \qquad i=1, \dots, N-1. 
$$
Our aim is to obtain $L^1-$compactness of $\rho^\ep$ far from $x^{\ep,N}(t)$. 

\begin{Proposition} \label{prop.localcompact} Fix $[t_1,t_2]\subset [0,\infty)$ and $[a_1,a_2]\subset \R$. Assume that
$$
N\ep \leq C_0,\qquad \|\rho^\ep\|_{L^\infty((t_1,t_2)\times (a_1,a_2))}\leq \rho_{\max}\qquad {\rm and}\qquad x^{\ep,N}(0)\geq a_2,
$$
for some constant $C_0$. Then, the family $(\rho^\ep)$ is relatively compact in $L^1((t_1,t_2)\times (a_1,a_2))$ and any cluster point $\rho$ as $\ep\to 0$ is an entropy solution to 
\be\label{eq.CLLLL}
\partial_t \rho+ \partial_x(f(\rho))= 0 \qquad \text{in}\; (t_1,t_2)\times (a_1,a_2).
\ee
\end{Proposition} 

In contrast with the results of \cite{rigorousLWR} recalled in Proposition \ref{prop.DiFR} above, the result does not require (nor establishes) BV bounds or a special behavior of $X^{\ep,N}$. It is of  local nature. 

\begin{proof} Let us set 
$$
u^\ep(t,x)= \int_{-\infty}^x \rho^\ep(t,y)dy. 
$$
We note for later use that 
\begin{align}\label{eg.uep}
u^\ep(t,x)&= \sum_{i=1}^{N-1} y^{\ep,i}(t)  \int_{-\infty}^x {\bf 1}_{[x^{\ep, i}(t), x^{\ep, i+1}(t))}(y)dy \notag \\
&= \sum_{i=1}^{N-1} (\ep(i-1) + y^{\ep,i}(t) (x-x^{\ep,i}(t))) {\bf 1}_{[x^{\ep, i}(t), x^{\ep, i+1}(t))}(x)+ N\ep {\bf 1}_{[x^{\ep, N}(t), \infty) }(x). 
\end{align}
As $(\rho^\ep)$ is bounded in $L^\infty((t_1,t_2)\times (a_1,a_2))$,  $(\rho^\ep)$ is relatively compact for the $L^\infty-$weak-* convergence. We denote again by $(\rho^\ep)$ a converging subsequence and let $\rho$ be its limit. We have to check that $\rho$ is  an entropy solution to \eqref{eq.CLLLL} and that $(\rho^\ep)$ converges to $\rho$ in $L^1((t_1,t_2)\times (a_1,a_2))$.  

The strategy of proof is  the following: we first show  that $u^\ep$ is uniformly continuous and bounded in $(t_1,t_2)\times (a_1,a_2)$, so that we can assume that $(u^\ep)$ converges along the same subsequence to a limit denoted by $u$. Note that $\partial_x u= \rho$. 
Next we   show that $u$ is a viscosity solution to the Hamilton-Jacobi equation 
\be\label{eq.HJJJJ}
\partial _t u + f(\partial_x u)=0\qquad \text{in}\; (t_1,t_2)\times (a_1,a_2). 
\ee
This convergence is related to similar results in \cite{FSZ18, FoSa20, FoImMo09}. By a classical argument (see \cite{CFM24} for the local version we use here, and the references therein), this implies that $\rho=\partial_x u$ is an entropy solution to \eqref{eq.CLLLL}. On the other hand we prove the existence of a constant $C$ such that, for any $\phi\in C^\infty_c((t_1,t_2)\times (a_1,a_2))$, 
\be\label{approxHJ}
\Bigl |\int_{t_1}^{t_2}\int_{a_1}^{a_2} (- u^\ep \partial_t \phi  + f(\partial_x u^\ep)\phi) \Bigr| \leq C\ep \| \phi\|_{W^{1,\infty}} \ep. 
\ee
Let $w=w(t,x,dz)$ be the Young measure associated to (a further subsequence of) $(\rho^\ep)$. Passing to the limit in \eqref{approxHJ} leads to the equality
$$
\int_{t_1}^{t_2}\int_{a_1}^{a_2} \int_{\R}(- u(t,x) \partial_t \phi(t,x)  + f(z)\phi(t,x)) w(t,x,dz)dxdt =0. 
$$
We then easily conclude from the uniform concavity of $f$ that  $w(t,x,dz)= \delta_{\rho(t,x)}$ a.e., and therefore that the convergence of $\rho^\ep$ to $\rho$ actually holds in $L^1$. \\

\noindent {\it Step 0: $u^\ep$ is uniformly Lipschitz continuous.} Note  first that $u^\ep$ is  continuous in time-space and piecewise $C^1$ in time. As $\rho^\ep$ is bounded by $\rho_{\max}$, $u^\ep$ is uniformly Lipschitz continuous in space in $(t_1,t_2)\times (a_1,a_2)$.  Fix $(t,x)\in (t_1,t_2)\times (a_1,a_2)$ such that $x\in (x^{\ep,i}(t), x^{\ep,i+1}(t))$ for some $i\in \{1, \dots, N-1\}$. Then, by \eqref{eg.uep},  
\begin{align*}
\partial_t u^\ep (t,x) & = \dot y^{\ep,i}(t) (x-x^{\ep,i}(t)) - y^{\ep,i}(t) \dot x^{\ep,i}(t) \\ 
&= -\ep^{-1} (y^{\ep,i}(t))^2 (x-x^{\ep,i}(t))(v(y^{\ep,i+1}(t))-v(y^{\ep,i}(t))) - y^{\ep,i}(t) v(y^{\ep,i}(t)), 
\end{align*}
so that 
\begin{align*}
|\partial_t u^\ep (t,x)| & \leq \ep^{-1} (y^{\ep,i}(t))^2 (x^{\ep,i+1}(t)-x^{\ep,i}(t)) |v(y^{\ep,i+1}(t))-v(y^{\ep,i}(t))| + y^{\ep,i}(t) v(y^{\ep,i}(t))\\ 
& \leq 3 y^{\ep,i}(t) V_{\max} \leq 3\rho_{\max} V_{\max}, 
\end{align*}
since $y^{\ep,i}(t) (x^{\ep,i+1}(t)-x^{\ep,i}(t))=\ep$ and $\|\rho^\ep\|_\infty\leq \rho_{\max}$. If $x< x^{\ep,i}(t)$ or $x> x^{\ep,N}(t)$, then $\partial_tu^\ep (t,x)=0$. This shows that $u^\ep$ is also uniformly Lipschitz continuous in time. \\ 

\noindent {\it Step 1: $u$ is a viscosity solution of \eqref{eq.HJJJJ}.} Let us consider 
$$
\tilde u^\ep(t,x) = \sum_{i=1}^{N-1} \ep(i-1)  {\bf 1}_{[x^{\ep, i}(t), x^{\ep, i+1}(t))}(x)+ N\ep {\bf 1}_{[x^{\ep, N}(t), \infty) }(x).
$$
Recalling \eqref{eg.uep}, we have 
\begin{align*}
|\tilde u^\ep(t,x) -u^\ep(t,x)| & \leq \sum_{i=1}^{N-1} y^{\ep,i}(t) (x-x^{\ep,i}(t))) {\bf 1}_{[x^{\ep, i}(t), x^{\ep, i+1}(t))}(x) \\ 
& \leq \sum_{i=1}^{N-1} y^{\ep,i}(t) (x^{\ep, i+1}(t) -x^{\ep,i}(t))) {\bf 1}_{[x^{\ep, i}(t), x^{\ep, i+1}(t))}(x) \leq \ep. 
\end{align*}
As $u^\ep$ converges uniformly to $u$, $\tilde u^\ep$ also converges uniformly to $u$. In order to prove that $u$ is a subsolution, let $\phi$ be a $C^2$ test function such that $u-\phi$ has a strict local maximum at some point $(\bar t, \bar x)\in (t_1,t_2)\times (a_1,a_2)$. We have to check that 
$$
\partial_t \phi(\bar t,\bar x) + f(\partial_x\phi(\bar t,\bar x)) \leq 0. 
$$
As $u$ is nondecreasing in $x$ and nonincreasing in $t$ (because this is the case for the $u^\ep$), we have $\partial_x\phi(\bar t, \bar x)\geq 0$ and  $\partial_t\phi(\bar t, \bar x)\leq 0$. If  $\partial_x\phi(\bar t, \bar x)= 0$, then 
$$
0\geq \partial_t \phi(\bar t, \bar x) = \partial_t \phi(\bar t,\bar x) + f(\partial_x\phi(\bar t,\bar x)), 
$$ 
since $f(0)=0$, and the claim is proved. We now assume that $\partial_x\phi(\bar t, \bar x)> 0$.  By standard argument in viscosity solutions, there exists $(t^\ep, x^\ep)$, which converges to $(\bar t, \bar x)$, such that $\tilde u^\ep-\phi$ has a local maximum at $(t^\ep, x^\ep)$. This means that, for $(t,x)$ close to $(t^\ep, x^\ep)$, 
\be\label{iakqzjehsdn}
u^\ep(t,x) \leq u^\ep(t^\ep, x^\ep)+ \phi(t,x)-\phi(t^\ep, x^\ep).
\ee
For $\ep$ small enough, $\partial_x\phi(t^\ep,x^\ep) >0$. Note then that there exists $i_0\in \{1, \dots, N\}$ such that $x^\ep=x^{\ep, i_0}(t^\ep)$, because $\tilde u^\ep$ is locally constant anywhere else. As $t\to \tilde u^\ep(t, x^{\ep, i_0}(t))$ is constant, we get by \eqref{iakqzjehsdn}: 
$$
0= \partial_t \phi(t^\ep, x^\ep)+\partial_x \phi(t^\ep, x^\ep) \dot x^{\ep, i_0}(t),
$$
and thus 
\be\label{lazkqejsdf}
0= \partial_t \phi(t^\ep, x^\ep)+\partial_x \phi(t^\ep, x^\ep) v(y^{\ep, i_0}(t^\ep)).
\ee
If, up to a subsequence, $y^{\ep, i_0}(t^\ep)$ tends to $0$, then, letting $\ep\to 0$ in the previous equality gives
$$
0 = \partial_t \phi(\bar t, \bar x)+\partial_x \phi(\bar t, \bar x) V_{\max} \geq \partial_t \phi(\bar t, \bar x)+\partial_x \phi(\bar t, \bar x) v(\partial_x \phi(\bar t, \bar x)) =   \partial_t \phi(\bar t, \bar x)+f(\partial_x \phi(\bar t, \bar x)),
$$
and the claim is proved. Otherwise, $y^{\ep, i_0}(t^\ep)$ is bounded below by a constant $K^{-1}$. Let us choose $t=t^\ep$ and $x=x^{\ep, i_0+1}(t^\ep)$ in \eqref{iakqzjehsdn}: as $u^\ep(t^\ep,x^{\ep, i_0+1}(t^\ep))= u^\ep(t^\ep,x^{\ep, i_0}(t^\ep))+\ep$,  we obtain 
\begin{align*}
\ep & \leq \phi(t^\ep, x^{\ep,i_0+1}(t^\ep))- \phi(t^\ep, x^{\ep,i_0}(t^\ep)) )\\
&  \leq   \partial_x\phi(t^\ep, x^{\ep,i_0}(t^\ep))(x^{i_0+1}(t^\ep) -x^{\ep,i_0}(t^\ep))+C (x^{\ep,i_0+1}(t^\ep) -x^{\ep,i_0}(t^\ep))^2, 
\end{align*}
where $C$ depends on $\|\partial_{xx}\phi\|_\infty$ only. Note that $\frac{\ep}{y^{\ep,i_0}(t^\ep)}= x^{\ep,i_0+1}(t^\ep) -x^{\ep,i_0}(t^\ep)\leq K\ep$. So 
$$
y^{\ep, i_0}(t^\ep) = \frac{\ep}{x^{\ep,i_0+1}(t^\ep) -x^{\ep,i_0}(t^\ep)}  \leq \partial_x\phi(t^\ep, x^{\ep,i_0}(t^\ep)) +CK\ep. 
$$
As $v$ is nonincreasing and locally Lipschitz in $[K^{-1}/2, \infty)$, we get 
$$
v(y^{\ep, i_0}(t^\ep)) \geq v ( \partial_x\phi(t^\ep, x^{\ep,i_0}(t^\ep))) - CK\ep.
$$
We plug this inequality into \eqref{lazkqejsdf} and then let $\ep\to 0$ to obtain the claim. \\ 

We now prove that $u$ is a supersolution. Let $\phi$ be a $C^2$ test function such that $u-\phi$ has a strict local minimum at some point $(\bar t, \bar x)\in (t_1,t_2)\times (x_1,x_2)$. We have to check that 
$$
\partial_t \phi(\bar t,\bar x) + f(\partial_x\phi(\bar t,\bar x)) \geq 0. 
$$
Let 
$$
\hat u^\ep(t,x) :=  \sum_{i=1}^{N-1} \ep(i-1)  {\bf 1}_{(x^{\ep, i}(t), x^{\ep, i+1}(t)]}(x)+ N\ep {\bf 1}_{(x^{\ep, N}(t), \infty) }(x).
$$
Note that $\hat u^\ep$ is the lower semi-continuous envelope of $\tilde u^\ep$ and thus  converges locally uniformly  to $u$. As $\hat u^\ep$  is  lower semi-continuous, for $\ep$ small enough,  there exists $(t^\ep, x^\ep)\in (t_1,t_2)\times (x_1,x_2)$, which converges to $(\bar t, \bar x)$, such that $\hat u^\ep-\phi$ has a local minimum at $(t^\ep, x^\ep)$. Then, for $(t,x)$ close to $(t^\ep, x^\ep)$,  
\be\label{ajklhzerdnf}
\hat u^\ep (t,x) \geq \hat u^\ep(t^\ep, x^\ep) +\phi(t,x)-\phi(t^\ep, x^\ep). 
\ee
If (up to a subsequence) $\hat u^\ep$ is locally constant near $(t^\ep, x^\ep)$, then $\partial_t \phi(t^\ep, x^\ep)= \partial_x \phi(t^\ep, x^\ep)=0$ and the claim holds. So we can assume that $\hat u^\ep$ is not locally constant near $(t^\ep, x^\ep)$, which implies the existence of $i_0\in \{1, \dots, N\}$ such that $x^\ep= x^{\ep, i_0}(t^\ep)$. Note that $i_0<N$ because $x^\ep < a_2$. As $t\to \hat u^\ep(t, x^{\ep, i_0}(t))$ is constant, we get 
$$
0= \partial_t \phi(t^\ep, x^\ep)+  \partial_x \phi(t^\ep, x^\ep)v(y^{\ep, i_0}(t^\ep)).
$$
We now note that, if $\partial_x \phi(\bar t, \bar x)=0$, then, as 
$$
0\leq  \partial_t \phi(t^\ep, x^\ep)+  \partial_x \phi(t^\ep, x^\ep)V_{\max}, 
$$
we get 
$$
0\leq  \partial_t \phi(\bar t, \bar x) = \partial_t \phi(\bar t, \bar x)+  f(\partial_x \phi(\bar t, \bar x)),
$$
since $f(0)=0$. Therefore the claim holds. We now suppose that $\partial_x \phi(\bar t, \bar x)>0$. We claim that $x^{\ep, i_0+1}(t^\ep)\leq x^{\ep, i_0}(t^\ep)+ \theta \ep$ (up to a subsequence), where $\theta = 2/\partial_x \phi(\bar t, \bar x)$. Indeed, otherwise, let us choose $t=t^\ep$ and $x= x^{\ep, i_0}(t^\ep)+ \theta \ep$ in \eqref{ajklhzerdnf}. Then 
$$
\ep \geq \phi(t^\ep, x^\ep + \theta \ep)- \phi(t^\ep, x^\ep). 
$$
Dividing by $\ep$ and letting $\ep\to0$ leads to 
$$
1 \geq \theta \partial_x\phi(\bar t,\bar x), 
$$
a contradiction with the definition of $\theta$. So $x^{\ep, i_0+1}(t^\ep)\leq x^{\ep, i_0}(t^\ep)+ \theta \ep$. Let us set $t=t^\ep$ and $x= x^{\ep, i_0+1}(t^\ep)$ in \eqref{ajklhzerdnf}. We get 
\begin{align*}
\ep & \geq \phi(t,x^{\ep, i_0+1}(t^\ep))-\phi(t^\ep, x^{\ep, i_0}(t^\ep))\\ 
& \geq \partial_x \phi(t^\ep, x^\ep) ( x^{\ep, i_0+1}(t^\ep)-x^{\ep, i_0}(t^\ep))- C( x^{\ep, i_0+1}(t^\ep)-x^{\ep, i_0}(t^\ep))^2.
\end{align*}
Hence, arguing as above, we obtain
$$
y^{\ep, i_0}(t^\ep) \geq \partial_x \phi(t^\ep, x^\ep)  -C \theta \ep,
$$
and thus 
$$
v(y^{\ep, i_0}(t^\ep)) \leq v(\partial_x \phi(t^\ep, x^\ep)) +C \theta \ep.
$$
We can then conclude as above. \\

\noindent {\it Step 2: $u^\ep$ is an approximation solution to the HJ equation, in the sense that \eqref{approxHJ} holds.} Fix $\phi\in C^\infty_c((t_1,t_2)\times (a_1,a_2), \R_+)$. We first use  \eqref{eg.uep} and the fact that $x^{\ep, N}(t)\geq x^{\ep, N}(0)\geq a_2$ by assumption and that $\phi$ has a support in $(t_1,t_2)\times (a_1,a_2)$ to obtain  
\begin{align*}
\int_{a_1}^{a_2} u^\ep \phi =   \sum_{i=1}^{N-1} \int_{x^{\ep,i}(t)}^{x^{\ep, i+1}(t)} (\ep(i-1) + y^{\ep,i}(t) (x-x^{\ep,i}(t))) \phi(t,x)dx.
\end{align*}
Therefore 
\begin{align*}
\frac{d}{dt} \int_{a_1}^{a_2} u^\ep \phi  & =  \int_{a_1}^{a_2} u^\ep \partial_t \phi+ \sum_{i=1}^{N-1 }   \int_{x^{\ep,i}(t)}^{x^{\ep, i+1}(t)} (\dot y^{\ep,i}(t) (x-x^{\ep,i}(t)) - y^{\ep,i}(t) 
\dot x^{\ep,i}(t)) \phi(t,x)dx \\
&   + \sum_{i=1}^{N-1 } \Bigl((\ep(i-1)+ y^{\ep,i}(t) (x^{\ep, i+1}(t)-x^{\ep,i}(t))) \phi(t,x^{\ep, i+1}(t))\dot  x^{\ep, i+1}(t) -  (\ep(i-1)\phi(t,x^{\ep,i}(t)))\dot x^{\ep, i}(t)\Bigr) \\ 
& =  \int_{a_1}^{a_2} u^\ep \partial_t \phi+ \sum_{i=1}^{N-1 }   \int_{x^{\ep,i}(t)}^{x^{\ep, i+1}(t)} (\dot y^{\ep,i}(t) (x-x^{\ep,i}(t)) - y^{\ep,i}(t) 
v(y^{\ep,i}(t)) \phi(t,x)dx \\
& \qquad  + \sum_{i=1}^{N-1 } \Bigl(\ep i \phi(t,x^{\ep, i+1}(t))\dot  x^{\ep, i+1}(t) -  \ep(i-1)\phi(t,x^{\ep,i}(t))\dot x^{\ep, i}(t)\Bigr) .
\end{align*}
Note that
$$
\sum_{i=1}^{N-1 } \Bigl(\ep i \phi(t,x^{\ep, i+1}(t))\dot  x^{\ep, i+1}(t) -  \ep(i-1)\phi(t,x^{\ep,i}(t))\dot x^{\ep, i}(t)\Bigr) = \ep (N-1) \phi(t,x^{\ep, N}(t))\dot  x^{\ep, N}(t) = 0
$$
since $x^{\ep, N}(t)\geq a_2$  and $\phi$ has a support in $(t_1,t_2)\times (a_1,a_2)$. Therefore 
\be\label{iuhnjfd}
\frac{d}{dt} \int_{a_1}^{a_2} u^\ep \phi   =  \int_{a_1}^{a_2} (u^\ep \partial_t \phi - f(\rho^\ep) \phi) +R(t),
\ee
where 
$$
R(t):= \sum_{i=1}^{N-1 }  \int_{x^{\ep,i}(t)}^{x^{\ep, i+1}(t)} \dot y^{\ep,i}(t) (x-x^{\ep,i}(t)) \phi(t,x)dx.
$$
We now estimate $R(t)$.  By integration by parts, we have 
\begin{align*}
R(t) &= \frac12  \sum_{i=1}^{N-1 } \Bigl( \dot y^{\ep,i}(t)  (x^{\ep,i+1}(t)-x^{\ep,i}(t))^2 \phi(t,x^{\ep,i+1}(t))   -   \dot y^{\ep,i}(t)  \int_{x^{\ep,i}(t)}^{x^{\ep, i+1}(t)}  (x-x^{\ep,i}(t))^2 \partial_x\phi (t,x)dx\Bigr) \\ 
& = - \frac{1}{2\ep} \sum_{i=1}^{N-1 }\Bigl(  (y^{\ep,i}(t))^2 (x^{\ep,i+1}(t)-x^{\ep,i}(t))^2 \phi(t,x^{\ep,i+1}(t))(v(y^{\ep, i+1}(t))-v(y^{\ep, i}(t))) \\
& \qquad  -(y^{\ep,i}(t) )^2 (v(y^{\ep, i+1}(t))-v(y^{\ep, i}(t))) \int_{x^{\ep,i}(t)}^{x^{\ep, i+1}(t)}  (x-x^{\ep,i}(t))^2 \partial_x\phi (t,x)dx\Bigr) \\ 
& = - \frac{\ep}{2} \sum_{i=1}^{N-1 }  \phi(t,x^{\ep,i+1}(t))(v(y^{\ep, i+1}(t))-v(y^{\ep, i}(t))) \\
& \qquad  +\frac{1}{2\ep} \sum_{i=1}^{N-1 }(y^{\ep,i}(t) )^2 (v(y^{\ep, i+1}(t))-v(y^{\ep, i}(t))) \int_{x^{\ep,i}(t)}^{x^{\ep, i+1}(t)}  (x-x^{\ep,i}(t))^2 \partial_x\phi (t,x)dx\\ 
&= I+II
\end{align*}
with 
\begin{align*}
 I  & =- \frac{\ep}{2}    \sum_{i=1}^{N-1 }  \phi(t,x^{\ep,i+1}(t))(v(y^{\ep, i+1}(t))-v(y^{\ep, i}(t)))  \\ 
& =  - \frac{\ep}{2}  \sum_{i=2}^{N-1 } v(y^{\ep, i}(t)) ( \phi(t,x^{\ep,i}(t))- \phi(t,x^{\ep,i+1}(t))) - \frac{\ep}{2} \phi(t,x^{\ep,N}(t))v(y^{\ep, N}(t))+\frac{\ep}{2} \phi(t,x^{\ep,2}(t))v(y^{\ep, 1}(t))   \\ 
& =   - \frac{\ep}{2}  \sum_{i=2}^{N-1 } v(y^{\ep, i}(t)) \int_{x^{\ep,i}(t)}^{x^{\ep, i+1}(t)} \partial_x \phi(t,x)dx+\frac{\ep}{2} \phi(t,x^{\ep,2}(t))v(y^{\ep, 1}(t))  ,
\end{align*}
so that 
\begin{align*}
 | I |  & \leq  \frac{V_{\max}\ep }{2} \|\partial_x\phi \|_\infty  \sum_{i=1}^{N-1 } \int_{x^{\ep,i}(t)}^{x^{\ep, i+1}(t)}{\bf 1}_{[a_1,a_2]}(x)dx +\ep\| \phi\|_\infty V_{\max}\\ 
&  \leq  V_{\max} \|\phi \|_{W^{1,\infty}} (a_2-a_1+1)  \ep , 
\end{align*}
while 
\begin{align*}
| II |  & =  \frac{1}{2\ep} \Bigl|  \sum_{i=1}^{N-1 } (y^{\ep,i}(t) )^2(v(y^{\ep, i+1}(t))-v(y^{\ep, i}(t))) \int_{x^{\ep,i}(t)}^{x^{\ep, i+1}(t)}  (x-x^{\ep,i}(t))^2 \partial_x\phi (t,x)dx\Bigr|  \\
& \leq V_{\max} \ep^{-1}\|\partial_x\phi \|_\infty \sum_{i=1}^{N-1 } (y^{\ep,i}(t) )^2 (x^{\ep,i+1}(t)-x^{\ep,i}(t))^2 \int_{x^{\ep,i}(t)}^{x^{\ep, i+1}(t)}{\bf 1}_{[a_1,a_2]}(x)dx \\
& \leq V_{\max} \ep\|\partial_x\phi \|_\infty  \int_{x^{\ep,1}(t)}^{x^{\ep, N}(t)}{\bf 1}_{[a_1,a_2]}(x)dx \leq V_{\max} \|\partial_x\phi\|_\infty  (a_2-a_1) \ep. 
\end{align*}
Integrating \eqref{iuhnjfd} over $[t_1,t_2]$ gives \eqref{approxHJ}. \\

\noindent {\it Step 3: Conclusion.}  Recall now that $w=w(t,x,dz)$ is the Young measure associated to (a further subsequence of) $(\rho^\ep)$. Note that 
$$
 \partial_x u(t,x)= \rho(t,x)= \int_{\R} w(t,x,dz).
 $$
Letting $\ep\to 0$ in \eqref{approxHJ} gives, for any test function $\phi\in C^\infty_c((t_1,t_2)\times (a_1,a_2))$: 
$$
\int_{t_1}^{t_2}\int_{a_1}^{a_2} \int_{\R}(- u(t,x) \partial_t \phi(t,x)  + f(z)\phi(t,x)) w(t,x,dz)dxdt =0. 
$$
On the other hand, $u$ being a Lipschitz continuous viscosity solution to \eqref{eq.HJJJJ}, it satisfies the equation a.e. and therefore
$$
\int_{t_1}^{t_2}\int_{a_1}^{a_2} \int_{\R}(- u(t,x) \partial_t \phi(t,x)  + f(\partial_x u(t,x))\phi(t,x))=0. 
$$
This implies that, for a.e. $(t,x)\in (t_1,t_2)\times (a_1,a_2)$, 
$$
f(\partial_x u(t,x))=  \int_{\R} f(z) w(t,x,dz). 
$$
As $f$ is strictly concave, we infer that $w(t,x,dz) = \delta_{\partial_x u(t,x)}(dz)$. Therefore $\rho^\ep=\partial_x u^\ep$ converges in $L^1$ to $\rho=\partial_x u$. 
\end{proof}

\section{Proof of the main result} \label{sec.proofmain}

The goal of the section is to prove Theorem \ref{thm.main}. From now on, we fix a solution $X=(X^i)_{i=1,\dots, N}$ to 
\eqref{eq.X}-\eqref{eq.init}-\eqref{eq.defri}. Recall that the traffic-light is $T$ periodic, green for Road $\mathcal R^1$ on intervals $[0, T_1)$ mod. $T$ and red otherwise. Moreover $\rho^\ep_X$ is defined by \eqref{def.rho}. We also assume that \eqref{Nep=O1}, \eqref{hyp.condinti}, \eqref{BVt=0}, \eqref{limt=0} and \eqref{hyp.T} hold. 

For simplicity of notation we prove the convergence on the scaled time interval $[0,1]$. 

 \subsection{BV estimates for the whole system}
 
 In order to control the $L^1$ distance between two discrete solutions (or a discrete solution and a continuous one), we need to bound the total variation of the discrete solution. For this we use the following convention: on a time interval of the form $(\ep nT, \ep(nT+T_1))$ (for $n\in \N$)---i.e., on which the traffic light is green for Road 1---we define 
 $$
 TV(v(\rho^\ep_X(t, \cdot)) = TV(v(\rho^{\ep,1}_X(t,\cdot) +\rho^{\ep,0}_X(t, \cdot)))+ TV (v(\rho^{\ep,2}_X(t,\cdot){\bf 1}_{(-\infty, x^{\ep,i_2}(t))})),
 $$
 where 
 $$
i_2= {\rm argmax}\{ X^i(t), \; r^i(t)=2\}.
$$
Note that this amount to consider $\mathcal R^1\cup \mathcal R^0$ as a single line and to set $\rho^{\ep,2}_X$ to zero on $[x^{\ep,i_2}(t), \infty)$, where $i_2$ is the index of the right-most vehicle on Road~2. On intervals of the form $(\ep (nT+T_1, \ep(n+1)T)$ (for $n\in \N$), we define  $TV(v(\rho^\ep_X(t, \cdot))$ symmetrically (looking at $\mathcal R^2\cup \mathcal R^0$ as a single line). This convention  introduces a  discontinuity of the total variation at the changes of traffic light. Its main interest is that it is compatible with the conventions used in Section \ref{sec.Esti1branch}.

 \begin{Proposition}\label{prop.gobBV}  There is a constant $K_0>0$, depending on $V$ and on $TV(v(\rho^\ep_X(0)))$ only, such that 
 $$
 TV(v(\rho^\ep_X(t, \cdot)) \leq K_0(1+ \ep^{-1}(T_1^{-1}\vee (T-T_1)^{-1})) \qquad \forall t\in [ 0,1]. 
 $$
 \end{Proposition} 

Note that the estimate degrades as $\ep\to 0$, because $\ep^{-1}(T_1^{-1}\vee (T-T_1)^{-1})$ tends to $\infty$, even under assumption \eqref{hyp.T}. 
 
 \begin{proof} We explain  the mechanism in detail at a change of traffic light at (physical) time $nT$ (with $n\geq 1$) and prove the estimate up to time $(nT+T_1)$. The case of the interval $[nT+T_1, (n+1)T)$ is symmetric.  Finally, we note that the case of the interval $[0,T_1]$ is just a propagation of the initial condition. 
 
 During the time interval $[(n-1)T+T_1, nT)$, the traffic flow on the road $\mathcal R^2\cup\mathcal R^0$ is unconstrained and we can apply the BV estimate  in Proposition \ref{prop.DiFR} to ensure that 
 \be\label{iqkushlkdjnf}
 TV(v(\rho^{\ep,2}_X(\ep^{-1} nT^-, \cdot)+ \rho^{\ep,0}_X(\ep^{-1} nT^-, \cdot))) \leq \bar C(1+ \ep^{-1}(T-T_1)^{-1}), 
 \ee
 where $\bar C$ depends on $V$ only (if $n=0$, $\bar C$ depends also on the BV assumption \eqref{BVt=0} on the initial condition).  Let $i_0$, $i_1$ and $i_2$ be the respective indices of the vehicle closest to the junction $x=0$ on road $\mathcal R^0$, $\mathcal R^1$ and $\mathcal R^2$  at time $nT$: 
 $$
 i_0=\text{argmin}_j \{ X^j(nT), \; r_j(nT)=0\}, \qquad i_k = \text{argmax}_j \{X^j(nT), \; r_j(nT)=k\}, \; k=1,2.
 $$
 As the position of the vehicles are distinct, if $i_0$, $i_1$ and $i_2$ exist, they are distinct. Note that in this case $\sigma^{i_1}(nT)= \sigma^{i_2}(nT)= i_0$. 
 Finally, let $j_1$ and $j_2$ be  the labels of the vehicle just behind $i_1$ and $i_2$ respectively (if such vehicle exist): $\sigma^{j_k}(nT)= i_k$ for $k=1,2$.  We assume here to fix the ideas that $i_0$, $i_1$, $i_2$, $j_1$ and $j_2$ exist. The case were one of them is not defined can be treated in a similar and simpler way. We finally assume that 
 \be\label{ouehizjrfdv}
 TV(v(\rho^{\ep,1}_X((\ep^{-1} nT)^-, \cdot){\bf 1}_{(-\infty, x^{i_1}(\ep^{-1} nT))})) \leq Q, 
 \ee
where $Q$ is defined below, and investigate what happens on the time interval $[nT, nT+T_1)$. Condition \eqref{ouehizjrfdv} is our induction condition.

We have
\begin{align*}
 TV(v(\rho^{\ep,2}_X(\ep^{-1} nT, \cdot){\bf 1}_{(-\infty, x^{i_2}(\ep^{-1} nT))})) & = TV(v(\rho^{\ep,2}_X(\ep^{-1} nT, \cdot)), (-\infty,x^{i_2}(\ep^{-1} nT))) +V(X^{i_2}(nT)-X^{j_2}(nT)) \\
&   \leq \bar C(1+ \ep^{-1}(T-T_1)^{-1}) +V_{\max}, 
\end{align*}
where we used \eqref{iqkushlkdjnf} to estimate $TV(\rho^{\ep,2}_X((\ep^{-1} nT)^-, \cdot), (-\infty, x^{i_2}(\ep^{-1} nT)))$. 
  By Proposition \ref{prop.BoundBVRed}, we obtain therefore
$$
TV(v(\rho^{\ep,2}_X(t, \cdot){\bf 1}_{(-\infty, x^{i_2}(\ep^{-1} t))} )) \leq \bar C(1+ \ep^{-1}(T-T_1)^{-1}) +5V_{\max}=: Q \qquad \forall t\in [nT, nT+T_1). 
$$
where we set $Q= \bar C(1+ \ep^{-1}(T_1^{-1}\vee (T-T_1)^{-1})) +5V_{\max} $.
On the other hand, 
 \begin{align*}
& TV (v(\rho^{\ep,1}_X(\ep^{-1}nT,\cdot)+ \rho^{\ep,0}_X(\ep^{-1}nT,\cdot))) \\
& = 
TV(v((\rho^{\ep,1}_X(\ep^{-1}nT,\cdot){\bf 1}_{(-\infty, x^{i_1}(\ep^{-1}nT)}) +TV(v(\rho^{\ep,0}_X(\ep^{-1}nT,\cdot)), (x^{i_0}(\ep^{-1},\infty)) + 2V_{\max}.
\end{align*} 
Thus, using \eqref{ouehizjrfdv} and \eqref{iqkushlkdjnf}, we obtain
 \begin{align*}
& TV (v(\rho^{\ep,1}_X(\ep^{-1}nT,\cdot)+ \rho^{\ep,0}_X(\ep^{-1}nT,\cdot))) \leq  Q + \bar C(1+ \ep^{-1}(T-T_1)^{-1}) +2V_{\max}. 
\end{align*} 
By Proposition \ref{prop.DiFR} we infer that, for any  $t\in [nT, nT+T_1)$, 
\begin{align*}
& TV (v(\rho^{\ep,1}_X(\ep^{-1}t,\cdot)+ \rho^{\ep,0}_X(\ep^{-1}t,\cdot))) \leq TV (v(\rho^{\ep,1}(\ep^{-1}nT,\cdot)+ \rho^{\ep,0}_X(\ep^{-1}nT,\cdot))\\
& \qquad  \leq 
  Q + \bar C(1+ \ep^{-1}(T-T_1)^{-1}) +2V_{\max}. 
\end{align*} 
We  infer that 
$$
TV(v(\rho^\ep_X(t, \cdot)) \leq K_0(1+ \ep^{-1}(T_1^{-1}\vee (T-T_1)^{-1}))  \qquad \forall t\in [\ep^{-1}nT,\ep^{-1}(nT +T_1)),
$$
for some constant $K_0$ independent of $n$. This inequality extends to the time interval $ [\ep^{-1}(nT +T_1), \ep^{-1}(n+1)T)$ by using symmetric argument. We can then conclude by induction. 
 \end{proof}

 \subsection{An approximate Kato's inequality}
 
We consider the solution $\rho^\ep_X$ defined by \eqref{def.rho} and a solution $\rho_Y$ to the following mesoscopic problem: 
\be\label{eq.mesobis}
\begin{array}{lllll}
(i)&\rho^j_Y\in [0,\rho_{\max}] &\qquad \text{ a.e. on}\; &(0,\infty)\times \mathcal R^j, & j=0,1,2\\
(ii)& \partial_t \rho^j_Y +\partial_x (f^j(\rho^j_Y))= 0 &\qquad \text{ on}\; &(0,\infty)\times \mathcal R^j, & j=0,1,2 \\
(iii) & (\rho^0_Y(t,0^+), \rho^1_Y(t,0^-), \rho^2_Y(t,0^-))\in \mathcal G^\ep(t) &\qquad \text{ for a.e.}\;  &t\in (0,\infty),&
\end{array}
\ee
where the maximal germ $\mathcal G^\ep$ is periodic in time, of period $\ep T$, and such that 
$$
{\mathcal G}^\ep(t):=\left\{\begin{array}{llll}
 {\mathcal G}^1 &\quad \mbox{on}\quad [0, \ep T_1]\; \mbox{\rm (mod. $\ep T$)}\\
 {\mathcal G}^2 &\quad \mbox{on}\quad [\ep T_1, \ep T]\; \mbox{\rm (mod. $\ep T$)}\\
\end{array}\right.
$$
with 
$$
{\mathcal G}^1= \{ (p^0,p^1,p^2)\in Q, \;   f^2(p^2)=0,\quad \min\left\{ f^{1,+}(p^1), f^{0,-}(p^0)\right\}= f^1(p^1)= f^0(p^0)\},$$
$${\mathcal G}^2= \{ (p^0,p^1,p^2)\in Q, \;   f^1(p^1)=0,\quad \min\left\{f^{2,+}(p^2), f^{0,-}(p^0)\right\}= f^2(p^2)= f^0(p^0)\}.
$$
This means that, on the time intervals $(0, \ep T_1)$ (mod. $\ep T$), $\rho^1_Y+\rho^0_Y$ solves the unconstrained equation 
\be\label{eq.unconst}
\partial_t  \rho +\partial_x (f^j(  \rho))= 0 \qquad x\in \R, 
\ee
while $\rho^2_Y$ solves the boundary value problem 
\be\label{eq.bvpb}
\begin{array}{l}
\rho(t,x) \in [0,\rho_{\max}]  \qquad x\leq 0\\
\partial_t \rho+\partial_x (f(\rho))= 0 \qquad x\leq 0\\ 
f(\rho(t,0^-))= 0 .
\end{array} 
 \ee
 On time intervals of the form $( \ep T_1,\ep T)$ (mod. $\ep T$), it is the opposite: $\rho^2_Y+\rho^0_Y$ solves the unconstrained equation \eqref{eq.unconst} while $\rho_Y^1$ solves the boundary value problem \eqref{eq.bvpb}. 
 
As above we fix $\delta:\R\to \R_+$ a smooth function, with support in $[-1,1]$ and integral $1$: $\int_\R \delta =1$, and set $\delta_h(x)= h\delta (hx)$ for $h>0$ large.

 \begin{Proposition}\label{prop.globalesti} For any $\psi\in C_c^\infty((0,1)\times \R, \R_+)$ and $h\geq 1$, we have
 \begin{align}\label{kj,hbn,lkjzsd}
&\sum_{k=0}^2  \iint\int_{0}^{1} \Bigl\{  | \rho^{\ep,k}_X(t,x)- \rho_Y^k(t,y)|  \psi_t(t, \frac{x+y}{2}) \delta_h(\frac{y-x}{2}) \notag\\ 
& \qquad + {\rm sgn}(  \rho^{\ep,k}_X(t,x)- \rho_Y^k(t,y)) (f( (t,x))-f( \rho_Y^k(t,y)))\psi_x(t, \frac{x+y}{2}) \delta_h(\frac{y-x}{2})
\Bigr\}\; dtdxdy  \notag \\ 
& \geq  \sum_{k=0}^2  \iint   | \rho^{\ep,k}_X(0,x)-\rho^k_Y(0,y)|  \psi(0, \frac{x+y}{2}) \delta_h(\frac{y-x}{2})dxdy  \\ 
& - C\ep K  \int_{0}^{1} (h^2 \|\psi(t, \cdot)\|_\infty+ h\|\psi_x(t, \cdot)\|_\infty)dt -C (h^{-1}+\ep) (n^\ep \|\psi\|_\infty + \int_0^{1} (\|\psi_t(t,\cdot)\|_\infty+
\|\psi_x(t,\cdot)\|_\infty) dt) \notag, 
\end{align}
where $n^\ep= [\ep^{-1}/T]$ and 
$$
K= \sup_ t TV(\rho^\ep_X(t,\cdot)).
$$
 \end{Proposition} 
 
 \begin{proof} We argue by induction on the time intervals $[\ep nT, \ep (n+1)T]$. To simplify the notation, we prove the result in the case $n=0$ (but without using the regularity of the initial condition). On $[0, T_1]$,  $\tilde \rho_Y:=\rho^1_Y+\rho^0_Y$ is a solution to \eqref{eq.unconst}, while $\tilde \rho^\ep_X:=\rho^{\ep,1}_X+\rho^{\ep,0}_X$ is of the form \eqref{def.rhoXDiFR} for an unconstrained solution to the follow-the-leader model. Hence, following Corollay \ref{cor.estiL1}, we have:  
 \begin{align} \label{lqkazjesnrdfg}
&\iint \Bigl\{  |\tilde \rho^\ep_X(\ep T_1^-,x)-\tilde \rho_Y(\ep T_1^-,y)|  \psi(\ep T_1, \frac{x+y}{2}) \delta_h(\frac{y-x}{2}) \Bigr\}\; dxdy \notag\\ 
&+ \iint\int_{0}^{\ep T_1} \Bigl\{  |\tilde \rho^\ep_X(t,x)-\tilde \rho_Y(t,y)|  \psi_t(t, \frac{x+y}{2}) \delta_h(\frac{y-x}{2}) \notag\\ 
& \qquad + {\rm sgn}(\tilde \rho^\ep_X(t,x)-\tilde \rho_Y(t,y)) (f(\tilde \rho^\ep_X(t,x))-f(\tilde \rho_Y(t,y)))\psi_x(t, \frac{x+y}{2}) \delta_h(\frac{y-x}{2})
\Bigr\}\; dtdxdy \notag \\ 
& \geq \iint \Bigl\{  |\tilde \rho^\ep_X(0,x)-\tilde \rho_Y(0,y)|  \psi(0, \frac{x+y}{2}) \delta_h(\frac{y-x}{2}) \Bigr\}\; dxdy
\\ 
& 
- C\ep \tilde K_0  \int_{0}^{\ep T_1} (h^2 \|\psi(t, \cdot)\|_\infty+ h\|\psi_x(t, \cdot)\|_\infty)dt , \notag
\end{align}
where 
$$
\tilde K_0= \sup_{t\in [0, \ep T_1)} TV(v(\tilde \rho^\ep_X(t))) = \sup_{t\in [0, \ep T_1)} TV(v(\rho^{\ep,1}_X(t)+\rho^{\ep,0}_X(t))).
$$
We need to write this expression in terms of $\rho^{\ep,k}_X$ and $\rho^{k}_Y$, $k=0,1$. For this, let us fix a time $t$ and let us denote by $i_0$, $i_1$ and $i_2$  the respective indices of the vehicle closest to the junction $x=0$ on road $\mathcal R^0$, $\mathcal R^1$ and $\mathcal R^2$  at time $\ep^{-1}t$: 
 $$
 i_0=\text{argmin}_j \{ X^j(\ep^{-1}t), \; r_j(\ep^{-1}t)=0\}, \qquad i_k = \text{argmax}_j \{X^j(\ep^{-1}t), \; r_j(\ep^{-1}t)=k\}, \; k=1,2.
 $$
We suppose that these indices exists, the case where one of them does not can be handled in a similar and simpler way. 
Let us consider a generic continuous map $g: \R^4\to \R$, with compact support in the first two variables. Typically, we will take below $g(x,y, a,b)= |b-a|\psi(\ep T_1, \frac{x+y}{2})$, or $g(x,y, a,b)= |b-a|\psi_t(t, \frac{x+y}{2})$, or 
$g(x,y,a,b)=  {\rm sgn}(b-a) (f(b)-f(a))\psi_x(t, \frac{x+y}{2})$,... We assume that
\be\label{hyp.gggg}
g(x,y,0,0)= 0\qquad \text{and}\qquad   |g(x,y,a,b)-g(x,y,a',b')|\leq C_g (|a'-a|+|b'-b|), 
\ee
for some constant $C_g$. We note that this is the case in the examples given above. Then, as $\tilde \rho^\ep_X=  \rho^{\ep,0}_X$ on $\{t\}\times[x^{\ep,i_0}, \infty)$ and $\tilde \rho^\ep_X=  \rho^{\ep,1}_X$ on $\{t\}\times(-\infty,x^{\ep,i_0})$ while 
 $\tilde \rho_Y= \rho_Y^0$ on $[0, \infty)$ and $\tilde \rho_Y= \rho_Y^1$ on $\{t\}\times(-\infty,0))$, we have 
\begin{align*}
& \iint g(x,y, \tilde \rho^\ep_X(t,x), \tilde \rho_Y(t,y))\delta_h(\frac{y-x}{2})  dxdy = 
 \iint g(x,y,  \rho^0_X(t,x),  \rho^0_Y(t,y))\delta_h(\frac{y-x}{2})  dxdy \\
 &\qquad \qquad  + \iint g(x,y,  \rho^1_X(t,x),  \rho^1_Y(t,y))\delta_h(\frac{y-x}{2})  dxdy + R^1+R^2
\end{align*}
where 
$$
R^1 = \int_{-\infty}^{x^{\ep,i_0}(t)} \int_0^\infty \Bigl( g(x,y,  \rho^1_X(t,x),  \rho^0_Y(t,y))  - 
g(x,y,  \rho^0_X(t,x),  \rho^0_Y(t,y))  - g(x,y,  \rho^1_X(t,x),  \rho^1_Y(t,y))   \Bigr)\delta_h(\frac{y-x}{2}) dxdy
$$
and 
$$
R^2 = \int_{x^{\ep,i_0}(t)}^\infty \int_{-\infty}^0 \Bigl( g(x,y,  \rho^0_X(t,x),  \rho^1_Y(t,y))  - 
g(x,y,  \rho^0_X(t,x),  \rho^0_Y(t,y))  - g(x,y,  \rho^1_X(t,x),  \rho^1_Y(t,y))   \Bigr)\delta_h(\frac{y-x}{2}) dxdy. 
$$
To estimate $R^1$, we recall that $\rho^0_X$ vanishes on $(-\infty, x^{\ep,i_0}(t))$ while $\rho^1_Y$ vanishes on $(0,\infty)$. Thus, by \eqref{hyp.gggg}, we obtain 
\begin{align*}
|R^1| &\leq  \int_{-\infty}^{x^{\ep,i_0}(t)} \int_0^\infty \Bigl( |g(x,y,  \rho^1_X(t,x),  \rho^0_Y(t,y))  - 
g(x,y,  0,  \rho^0_Y(t,y))| +| g(x,y,  \rho^1_X(t,x),  0) -g(x,y, 0,0)|  \Bigr)\delta_h(\frac{y-x}{2}) dxdy \\ 
& \leq  2 C_g \int_{-\infty}^{x^{\ep,i_0}(t)} \int_0^\infty |\rho^1_X(t,x)| \delta_h(\frac{y-x}{2}) dxdy \; =\;  2 C_g \int_{-2h^{-1}}^{x^{\ep,i_0}(t)}  |\rho^1_X(t,x)|dx , 
\end{align*}
where we used in the last inequality the fact that $\delta_h$ has a support in $[-h^{-1}, h^{-1}]$. By Lemma \ref{lem.x+1-xpetit} (see also the proof of Proposition \ref{prop.x+1-xpetit}) we have $\rho^1_X\leq \rho_{\max}$ in $(-\infty, x^{\ep,1}(t))$ and $\rho^1_X= \ep/(x^{\ep,i_0}(t)-x^{\ep, 1}(t))$ on $(x^{\ep, 1}(t),x^{\ep,i_0}(t))$. Thus 
\begin{align*}
\int_{-2h^{-1}}^{x^{\ep,i_0}(t)}  |\rho^1_X(t,x)|dx  &= \int_{-2h^{-1}}^{ (-2h^{-1})\vee x^{\ep,1}(t)}  |\rho^1_X(t,x)|dx +\int_{(-2h^{-1})\vee x^{\ep,1}(t) }^{x^{\ep,i_0}(t)}  |\rho^1_X(t,x)|dx \\ 
 & \leq  2h^{-1} \rho_{\max} +\frac{ \ep}{x^{\ep,i_0}(t)-x^{\ep, 1}(t)} (x^{\ep,i_0}(t)-x^{\ep, 1}(t)) \leq C(h^{-1}+\ep). 
\end{align*}
For $R^2$ we recall that $\rho^1_X$ vanishes on $(x^{\ep,i_0}(t),\infty)$ while $\rho^0_Y$ vanishes on $(-\infty,0)$ and we obtain 
\begin{align*}
R^2 & \leq  \int_{x^{\ep,i_0}(t)}^\infty \int_{-\infty}^0 \Bigl( | g(x,y,  \rho^0_X(t,x),  \rho^1_Y(t,y))  -  g(x,y,  0,  \rho^1_Y(t,y))|+
|g(x,y,  \rho^0_X(t,x), 0)- g(x,y, 0, 0)|    \Bigr)\delta_h(\frac{y-x}{2}) dxdy \\ 
& \leq 2C_g \int_{x^{\ep,i_0}(t)}^\infty \int_{-\infty}^0 \rho^0_X(t,x)\delta_h(\frac{y-x}{2}) dxdy   \leq 2C_g \int_{x^{\ep,i_0}(t)}^{2h^{-1}}  \rho^0_X(t,x)dx \leq 
4h^{-1} C_g\rho_{\max} , 
\end{align*} 
since $\rho^0_X$ is bounded by $\rho_{\max}$ in $(x^{\ep, 0}(t), \infty)$ (see again the proof of Proposition \ref{prop.x+1-xpetit}). This proves therefore that
\begin{align*}
& \left| \iint g(x,y, \tilde \rho^\ep_X(t,x), \tilde \rho_Y(t,y))\delta_h(\frac{y-x}{2})  dxdy- \sum_{k=0,1}  \iint g(x,y,  \rho^k_X(t,x),  \rho^k_Y(t,y))\delta_h(\frac{y-x}{2})  dxdy \right| \leq CC_g (h^{-1}+\ep).
\end{align*}
Applying this estimate to the various terms in \eqref{lqkazjesnrdfg}, we get therefore 
 \begin{align*} 
&\sum_{k=0,1} 
\iint   |  \rho^{\ep,k}_X(\ep T_1^-,x)- \rho_Y^k(\ep T_1,y)|  \psi(\ep T_1, \frac{x+y}{2}) \delta_h(\frac{y-x}{2}) \; dxdy \notag\\ 
&+\sum_{k=0,1}  \iint\int_{0}^{\ep T_1} \Bigl\{  |  \rho^{\ep,k}_X(t,x)- \rho_Y^k(t,y)|  \psi_t(t, \frac{x+y}{2}) \delta_h(\frac{y-x}{2}) \notag\\ 
& \qquad + {\rm sgn}(  \rho^{\ep,k}_X(t,x)- \rho_Y^k(t,y)) (f(  \rho^{\ep,k}_X(t,x))-f( \rho_Y^k(t,y)))\psi_x(t, \frac{x+y}{2}) \delta_h(\frac{y-x}{2})
\Bigr\}\; dtdxdy \notag \\ 
& \geq  \sum_{k=0,1}
 \iint   |  \rho^{\ep,k}_X(0,x)- \rho_Y^k(0,y)|  \psi(0, \frac{x+y}{2}) \delta_h(\frac{y-x}{2}) \; dxdy \\ 
& 
- C\ep \tilde K_0  \int_{0}^{\ep T_1} (h^2 \|\psi(t, \cdot)\|_\infty+ h\|\psi_x(t, \cdot)\|_\infty)dt -C (h^{-1}+\ep) (\|\psi\|_\infty + \int_0^{\ep T_1} (\|\psi_t(t,\cdot)\|_\infty+
\|\psi_x(t,\cdot)\|_\infty) dt) . 
\end{align*}
 
 We now turn to the stopped traffic on Road 2. Note that the index $i_2$ defined above remains unchanged on the time interval $(0, \ep T_1)$. Moreover,  
 $\rho^2_Y$ solves \eqref{eq.bvpb} while, if we set 
$$
\tilde \rho^2_X(t,x)=  \sum_{x^{\ep,i}(t)<0, \; r^i(t) =2, \; i\neq i_2} y^{\ep, i}(t)  {\bf 1}_{[x^i(t), x^{\ep,\sigma^i(\ep^{-1}t)}(t))},
$$
then $\tilde \rho^2_X$ corresponds to the density \eqref{def.rhoXStopped} of a stopped traffic as in Subsection \ref{subsec.stoppedtraf}.  We note for later use that 
\be\label{iakzjs;dkfg}
\rho^2_X(t,x)-\tilde \rho^2_X(t,x)= y^{\ep, i_2}(t)  {\bf 1}_{[x^{\ep, i_2}(t), x^{\ep,i_0}(t))}.
\ee
By Proposition \ref{prop.rateFeu}, we have 
 \begin{align}\label{kjaenzrekl}
&\iint \Bigl\{  |\tilde \rho^2_X(\ep T_1^-,x)-\rho^2_Y(\ep T_1,y)|  \psi(\ep T_1, \frac{x+y}{2}) \delta_h(\frac{y-x}{2})\Bigr\}dxdy \notag \\ 
&+ \iint \int_{0}^{\ep T_1}  \Bigl\{  |\tilde \rho^2_X(t,x)-\rho^2_Y(t,y)|  \psi_t(t, \frac{x+y}{2}) \delta_h(\frac{y-x}{2}) \notag \\ 
& \qquad + {\rm sgn}(\tilde \rho^2_X(t,x)-\rho^2_Y(t,y)) (f(\tilde \rho^2_X(t,x))-f(\rho^2_Y(t,y))) \psi_x(t, \frac{x+y}{2}) \delta_h(\frac{y-x}{2})
\Bigr\}\; dtdxdy \notag \\ 
& \geq \iint \Bigl\{  |\tilde \rho^2_X(0,x)-\rho^2_Y(0,y)|  \psi(0, \frac{x+y}{2}) \delta_h(\frac{y-x}{2})\Bigr\}dxdy  \\ 
& \qquad - C \int_{0}^{\ep T_1} (\ep \tilde K_2 h^2 \|\psi(t, \cdot)\|_\infty+ (\ep \tilde K_2 h+h^{-1})\|\psi_x(t, \cdot)\|_\infty +
 (\ep +h^{-1})\|\psi_t(t, \cdot)\|_\infty  )dt, \notag
\end{align}
where 
$$
\tilde K_2= \sup_{t\in [0, \ep T_1)} TV(\tilde  \rho^{\ep,2}_X(t))= \sup_{t\in [0, \ep T_1)} TV(  \rho^{\ep,2}_X(t){\bf 1}_{(\infty, x^{\ep, 2}(t))}). 
$$
We now need to replace $\tilde \rho^2_X$ by $\rho^2_X$.
In order to handle the different terms in the expression above, we introduce as before a generic continuous map $g: \R^4\to \R$, with compact support in the first two variables and assume that \eqref{hyp.gggg} holds. Then, by \eqref{iakzjs;dkfg}, 
\begin{align*}
& \left| \iint g(x,y,\tilde \rho^2_X(t,x),\rho^2_Y(t,y)) \delta_h(\frac{y-x}{2})  dxdy -\iint g(x,y, \rho^2_X(t,x),\rho^2_Y(t,y)) \delta_h(\frac{y-x}{2}) dxdy\right| \\
& \leq C_g \iint y^{\ep, i_2}(t)  {\bf 1}_{[x^{\ep, i_2}(t), x^{\ep,i_0}(t))} \delta_h(\frac{y-x}{2})dxdy = C_g \ep. 
\end{align*}
Applying this inequality to the various terms in \eqref{kjaenzrekl} we find therefore 
\begin{align*}
&\iint \Bigl\{  | \rho^2_X(\ep T_1^-,x)-\rho^2_Y(\ep T_1,y)|  \psi(\ep T_1, \frac{x+y}{2}) \delta_h(\frac{y-x}{2})\Bigr\}dxdy \notag \\ 
&+ \iint \int_{0}^{\ep T_1}  \Bigl\{  | \rho^2_X(t,x)-\rho^2_Y(t,y)|  \psi_t(t, \frac{x+y}{2}) \delta_h(\frac{y-x}{2}) \notag \\ 
& \qquad + {\rm sgn}( \rho^2_X(t,x)-\rho^2_Y(t,y)) (f( \rho^2_X(t,x))-f(\rho^2_Y(t,y))) \psi_x(t, \frac{x+y}{2}) \delta_h(\frac{y-x}{2})
\Bigr\}\; dtdxdy \notag \\ 
& \geq \iint \Bigl\{  | \rho^2_X(0^+,x)-\rho^2_Y(0,y)|  \psi(0, \frac{x+y}{2}) \delta_h(\frac{y-x}{2})\Bigr\}dxdy  \\ 
& \qquad - C\ep \|\psi\|_\infty- C \int_{0}^{\ep T_1} (\ep \tilde K_2 h^2 \|\psi(t, \cdot)\|_\infty+ (\ep \tilde K_2h+h^{-1}+\ep)\|\psi_x(t, \cdot)\|_\infty +
 (\ep +h^{-1})\|\psi_t(t, \cdot)\|_\infty  )dt. \notag
\end{align*}
Putting the estimates on all the lines together gives: 
 \begin{align} \label{kjaenzreklBIS}
&\sum_{k=0}^2 
\iint   |  \rho^{\ep,k}_X(\ep T_1^-,x)- \rho_Y^k(\ep T_1,y)|  \psi(\ep T_1, \frac{x+y}{2}) \delta_h(\frac{y-x}{2}) \; dxdy \notag\\ 
&+\sum_{k=0}^2  \iint\int_{0}^{\ep T_1} \Bigl\{  |  \rho^{\ep,k}_X(t,x)- \rho_Y^k(t,y)|  \psi_t(t, \frac{x+y}{2}) \delta_h(\frac{y-x}{2}) \notag\\ 
& \qquad + {\rm sgn}(  \rho^{\ep,k}_X(t,x)- \rho_Y^k(t,y)) (f(  \rho^{\ep,k}_X(t,x))-f( \rho_Y^k(t,y)))\psi_x(t, \frac{x+y}{2}) \delta_h(\frac{y-x}{2})
\Bigr\}\; dtdxdy \notag \\ 
& \geq  \sum_{k=0}^2
 \iint   |  \rho^{\ep,k}_X(0^+,x)- \rho_Y^k(0,y)|  \psi(0, \frac{x+y}{2}) \delta_h(\frac{y-x}{2}) \; dxdy \\ 
& 
- C\ep K  \int_{0}^{\ep T_1} (h^2 \|\psi(t, \cdot)\|_\infty+ h\|\psi_x(t, \cdot)\|_\infty)dt -C (h^{-1}+\ep) (\|\psi\|_\infty + \int_0^{\ep T_1} (\|\psi_t(t,\cdot)\|_\infty+
\|\psi_x(t,\cdot)\|_\infty) dt) \notag
\end{align}
where $K$ is defined as in the Proposition.  
We finally need to replace $\ep T_1^-$ by $\ep T_1^+$ in the left-hand side of the inequality above, since the $ \rho^{\ep,k}_X$ are discontinuous in time in general. Such a discontinuity can happen only if one of the vehicles reaches the position $x=0$ at time $t= \ep T_1$. We note  that  this can only  be a vehicle on road 1 which reaches $0$, i.e., $x^{\ep, i_1}(\ep T_1)=0$. Thus, in this case,   
$$
\rho^{\ep,0}_X(\ep T_1^+,x)= \rho^{\ep,0}_X(\ep T_1^-,x)+ \frac{\ep}{ x^{\ep, i_0}(\ep T_1)-x^{\ep, i_1}(\ep T_1)}{\bf 1}_{[ x^{\ep, i_1}(\ep T_1),x^{\ep, i_0}(\ep T_1) )}
$$
$$
\rho^{\ep,1}_X(\ep T_1^+,x)= \rho^{\ep,1}_X(\ep T_1^-,x)- \frac{\ep}{x^{\ep, i_0}(\ep T_1)-x^{\ep, i_1}(\ep T_1)} {\bf 1}_{[ x^{\ep, i_1}(\ep T_1),x^{\ep, i_0}(\ep T_1) )},
$$
\begin{align*}
\rho^{\ep,2}_X(\ep T_1^+,x) & = \rho^{\ep,2}_X(\ep T_1^-,x)- \frac{\ep}{x^{\ep, i_0}(\ep T_1)-x^{\ep, i_2}(\ep T_1)}{\bf 1}_{[ x^{\ep, i_2}(\ep T_1), x^{\ep, i_0}(\ep T_1))}\\ 
& + \frac{\ep}{x^{\ep, i_1}(\ep T_1)-x^{\ep, i_2}(\ep T_1)}{\bf 1}_{[ x^{\ep, i_2}(\ep T_1), x^{\ep, i_1}(\ep T_1))} . 
\end{align*}
Therefore, in any case, 
$$
\|\rho^{\ep,k}_X(\ep T_1^+,\cdot)-\rho^{\ep,k}_X(\ep T_1^-,\cdot)\|_{L^1} \leq C\ep. 
$$
Thus we can replace $\ep T_1^-$ by $\ep T_1^+$ in the left-hand side of \eqref{kjaenzreklBIS}, up to an increase of the constant $C$ in the right-hand side. 
We can obtain a similar inequality on the time interval $(\ep T_1, \ep T)$ and then complete the proof by induction on the time intervals $(n\ep T, (n+1)\ep T)$: the number $n^\ep$ in \eqref{kj,hbn,lkjzsd} is the number of switches of the traffic light on the scaled  time interval $[0,1]$.  
\end{proof}

%
%

\subsection{Correctors for the mesoscopic model} 

The convergence of the microscopic model to the continuous one is based on a  mesoscopic model \cite{CFM24}, which is  a traffic flow model with a periodic traffic light and in which the flow of vehicles is given by a continuous model of conservation law. In this model, the time-interval $\R$ is split into the $1-$periodic sets $I^1$ and $I^2$, with $I^1=[0,\theta]$ (mod. 1) and $I^1=[\theta,1]$ mod. 1 (with $\theta\in (0,1)$). On the time-intervals $I^1$, only cars coming from Road $1$ are allowed to enter the junction and the Road $0$, while on the time-intervals $I^2$ only cars coming from road $2$ can enter road $0$. There is no flux limiter on the junction. 

One of the main results of \cite{CFM24} is the existence of a time periodic corrector for the corresponding conservation law: 
\be\label{eq.meso}
\begin{array}{lllll}
(i)&\rho^j\in [0,\rho_{\max}] &\qquad \text{ a.e. on}\; &(-\infty,\infty)\times \mathcal R^j, & j=0,1,2\\
(ii)& \partial_t \rho^j +\partial_x (f^j(\rho^j))= 0 &\qquad \text{ on}\; &(-\infty,\infty)\times \mathcal R^j, & j=0,1,2 \\
(iii) & (\rho^0(t,0^+), \rho^1(t,0^-), \rho^2(t,0^-))\in \mathcal G(t) &\qquad \text{ for a.e.}\;  &t\in (-\infty,\infty),&
\end{array}
\ee
The time periodic maximal germ ${\mathcal G}$ of period equal to $1$ is given by
\begin{equation}\label{eq::c21bis}
{\mathcal G}(t):=\left\{\begin{array}{llll}
 {\mathcal G}^1(t) &\quad \mbox{on}\quad I^1\\
 {\mathcal G}^2(t) &\quad \mbox{on}\quad I^2\\
\end{array}\right.
\end{equation}
and 
$$
{\mathcal G}^1(t)= \{ (p^0,p^1,p^2)\in Q, \;   f^2(p^2)=0,\quad \min\left\{ f^{1,+}(p^1), f^{0,-}(p^0)\right\}= f^1(p^1)= f^0(p^0)\},$$
$${\mathcal G}^2(t)= \{ (p^0,p^1,p^2)\in Q, \;   f^1(p^1)=0,\quad \min\left\{f^{2,+}(p^2), f^{0,-}(p^0)\right\}= f^2(p^2)= f^0(p^0)\}.
$$

\begin{Theorem} {\bf (Existence of correctors with prescribed values at infinity) \cite{CFM24}}\label{thm.corrector}\\ 
Assume our standing assumptions {\bf (H)} hold and let $E$ be defined by \eqref{def.E}. For any $p=(p^0,p^1,p^2)\in E$, there exists an entropy solution $u_p=(u_p^k)_{k=0,1,2}\in L^\infty(\R\times \mathcal R)$ of \eqref{eq.meso} which is $1$-periodic in time and a constant $C>0$ such that for all $M\geq C$
\be\label{keypptcorr}
 \|u^0_p-p^0\|_{L^\infty(\R\times (M,\infty))}+ \|u^k_p-p^k\|_{L^\infty(\R\times (-\infty,M))}\leq CM^{-1},\qquad  k=1,2. 
\ee
\end{Theorem} 

\subsection{Proof of the convergence} 

We now prove Theorem \ref{thm.main}. Let us start with a compactness result. 

\begin{Lemma} The family $(\rho^\ep_X)$ is relatively compact in $L^1((0,1)\times \mathcal R)$. 
\end{Lemma} 

\begin{proof}  This is a simple application of Proposition \ref{prop.localcompact} in the interior of each branch, combined with the fact that $\rho^\ep$ is bounded, outside of a small neighborhood of $0$, where its mass is small (see Proposition \ref{prop.x+1-xpetit}).
\end{proof}

We now consider a converging subsequence of $(\rho^\ep_X)$, denoted again $(\rho^\ep_X)$ for simplicity. Let $\rho$ be its limit. We claim that $\rho$ is the unique solution to \eqref{eq.macro}, which will prove the convergence of the whole compact sequence: 

\begin{Lemma} The limit $\rho$ is an entropy solution to \eqref{eq.macro}, where the germ $\mathcal G$ is given by \eqref{def.mathG}. 
\end{Lemma}

\begin{proof} The fact that $\rho$ is an entropy solution outside $x=0$ is proved in Proposition \ref{prop.localcompact}. Recalling Remark  \ref{rem.entropsol}, we just need to check the following entropy inequality:
\be\label{ineq.des}
\sum_{k=0}^2\left\{ \int_0^\infty \int_{\mathcal R^k} \eta(p^k,\rho^k)\partial_t \psi^k+ q^k(p^k,\rho^k) \partial_x\psi^k +\int_{\mathcal R^k} \eta(p^k,\bar \rho^k)\psi^k(0,x)\right\} \geq 0
\ee
for any $p=(p^k) \in E$ and any continuous nonnegative test function $\psi:[0,\infty)\times \mathcal R\to [0,\infty)$ with a compact support and such that $\psi^k:=\psi_{|[0,+\infty)\times (\mathcal R^k\cup \left\{0\right\})}$ is $C^1$ for any $k=0,1,2$. Fix $p\in E$ and $\psi$ as above. According to Theorem \ref{thm.corrector}, there exists a solution $(u_p^k)\in L^\infty(\R\times \mathcal R)$ of \eqref{eq.meso} which is $1$-periodic in time satisfying \eqref{keypptcorr}. Let us define the scaled solution
$$
u^{\ep,k}_p(t,x)= u^k_p((\ep T)^{-1} t, (\ep T)^{-1} x), 
$$
which is an $\ep T-$periodic in time solution to \eqref{eq.mesobis}. By Proposition  \ref{prop.globalesti} we have 
\begin{align}\label{inj;ezsr dtfg:}
&\sum_{k=0}^2  \iint\int_{0}^{1} \Bigl\{  |  \rho^{\ep,k}_X(t,x)- u_p^{\ep,k}(t,y)|  \psi_t(t, \frac{x+y}{2}) \delta_h(\frac{y-x}{2}) \notag\\ 
& \qquad + {\rm sgn}(  \rho^{\ep,k}_X(t,x)- u_p^{\ep,k}(t,y)) (f(  \rho^{\ep,k}_X(t,x))-f( u_p^{\ep,k}(t,y)))\psi_x(t, \frac{x+y}{2}) \delta_h(\frac{y-x}{2})
\Bigr\}\; dtdxdy \notag \\ 
& \geq  \sum_{k=0}^2  \iint   | \rho^{\ep,k}_X(0^+,x)-u^{\ep,k}_p(0,y)|  \psi(0, \frac{x+y}{2}) \delta_h(\frac{y-x}{2})dxdy  \\ 
& 
 - C\ep K  \int_{0}^{1} (h^2 \|\psi(t, \cdot)\|_\infty+ h\|\psi_x(t, \cdot)\|_\infty)dt -C (h^{-1}+\ep) (n^\ep \|\psi\|_\infty + \int_0^{1} (\|\psi_t(t,\cdot)\|_\infty+
\|\psi_x(t,\cdot)\|_\infty) dt) \notag, 
\end{align}
where $n^\ep= [\ep^{-1}/T]$ and, by Proposition \ref{prop.gobBV},
$$
K= \sup_{t\in [0,1]} TV(v(\rho^\ep_X(t, \cdot)) \leq K_0(1+ \ep^{-1}(T_1^{-1}\vee (T-T_1)^{-1})) .
$$
Recall that, by assumption \eqref{hyp.T}, $T=\ep^{-\alpha}$ and  $T_1= \theta T$ with $\theta\in (0,1)$. If we choose $h= \ep^{-1/3}$, we get 
\begin{align*}
C\ep K  \int_{0}^{1} (h^2 \|\psi(t, \cdot)\|_\infty+ h\|\psi_x(t, \cdot)\|_\infty)dt +C (h^{-1}+\ep) (n^\ep \|\psi\|_\infty + \int_0^{1} (\|\psi_t(t,\cdot)\|_\infty+
\|\psi_x(t,\cdot)\|_\infty) dt) \leq C \ep^{\alpha-2/3}, 
\end{align*}
where, here and below, $C$ depends on $\|\psi\|_{W^{1,\infty}}$. Fix $M>0$ large. Then, by  \eqref{keypptcorr}, we have 
\begin{align*}
& \sum_{k=0}^2  \iint\int_{0}^{1} \Bigl\{  |  \rho^{\ep,k}_X(t,x)- u_p^{\ep,k}(t,y)|  \psi_t(t, \frac{x+y}{2}) \delta_h(\frac{y-x}{2})dxdy \\
&\qquad  \leq \sum_{k=0}^2  \iint\int_{0}^{1} \Bigl\{  |  \rho^{\ep,k}_X(t,x)- p^k|  \psi_t(t, \frac{x+y}{2}) \delta_h(\frac{y-x}{2})dxdy + 
C(M^{-1}+\ep).
\end{align*}
A symmetric inequality holds at time $t=0$. 
Handling all the other terms in \eqref{inj;ezsr dtfg:} in the same way, we get 
\begin{align*}
&\sum_{k=0}^2  \iint\int_{0}^{1} \Bigl\{  |  \rho^{\ep,k}_X(t,x)- p^k|  \psi_t(t, \frac{x+y}{2}) \delta_h(\frac{y-x}{2}) \notag\\ 
& \qquad + {\rm sgn}(  \rho^{\ep,k}_X(t,x)- p^k) (f(  \rho^{\ep,k}_X(t,x))-f( p^k))\psi_x(t, \frac{x+y}{2}) \delta_h(\frac{y-x}{2})
\Bigr\}\; dtdxdy  \\ 
& \geq  \sum_{k=0}^2  \iint   | \rho^{\ep,k}_X(0^+,x)-p^k|  \psi(0, \frac{x+y}{2}) \delta_h(\frac{y-x}{2})dxdy - C(\ep^{\alpha-2/3}+ M^{-1} ). \notag
\end{align*}
The $L^1$ convergence of $\rho^{\ep}$  (and \eqref{limt=0} for time $t=0$) and the assumption $\alpha\in (2/3,1)$ then gives \eqref{ineq.des}. 
\end{proof}

\paragraph{\textbf{Acknowledgement.}}
This research was partially funded by l'Agence Nationale de la Recherche (ANR), project ANR-22-CE40-0010 COSS. 
For the purpose of open access, the authors have applied a CC-BY public copyright licence to any Author Accepted Manuscript (AAM) version arising from this submission. The author thanks R\'egis Monneau for fruitful discussion.

\end{document}